\documentclass[11pt]{article}

\usepackage[vcentermath,enableskew]{youngtab}
\usepackage{amssymb}
\usepackage{amsthm}
\usepackage{amsmath}
\allowdisplaybreaks

\newtheorem{theorem}{Theorem}[section]
\newtheorem{lemma}{Lemma}[section]
\newtheorem{proposition}{Proposition}[section]
\newtheorem{corollary}{Corollary}[section]

\numberwithin{equation}{section}

\newcommand{\blam}{\boldsymbol{\lambda}}
\newcommand{\C}{\mathbb{C}}
\newcommand{\bT}{\boldsymbol{T}}
\newcommand{\bone}{\boldsymbol{1}}
\newcommand{\hgt}{\mathrm{ht}}
\newcommand{\Ind}{\mathrm{Ind}}
\newcommand{\Res}{\mathrm{Res}}
\newcommand{\FF}{\mathbb{F}}  

\newcommand{\GL}{\mathrm{GL}}

\newcommand{\wt}{\mathrm{wt}}

\newcommand{\ch}{\mathrm{ch}}
\newcommand{\spn}{\text{-}\mathrm{span}}
\newcommand{\mcX}{\mathcal{X}}
\newcommand{\mcS}{\mathcal{S}}

\newcommand{\mcO}{\mathcal{O}}

\newcommand{\mcP}{\mathcal{P}}
\newcommand{\tlam}{\tilde{\lambda}}
\newcommand{\Irr}{\mathrm{Irr}}

\def\adots{\mathinner{\mkern2mu\raise0pt\hbox{.}  
\mkern2mu\raise4pt\hbox{.}\mkern1mu
\raise7pt\vbox{\kern7pt\hbox{.}}\mkern1mu}}

\begin{document}

\bibliographystyle{amsplain}

\title{Kostka multiplicity one for multipartitions}
\author{James Janopaul-Naylor and C. Ryan Vinroot}
\date{}

\maketitle

\begin{abstract}
If $[\lambda(j)]$ is a multipartition of the positive integer $n$ (a sequence of partitions with total size $n$), and $\mu$ is a partition of $n$, we study the number $K_{[\lambda(j)]\mu}$ of sequences of semistandard Young tableaux of shape $[\lambda(j)]$ and total weight $\mu$.  We show that the numbers $K_{[\lambda(j)] \mu}$ occur naturally as the multiplicities in certain permutation representations of wreath products.  The main result is a set of conditions on $[\lambda(j)]$ and $\mu$ which are equivalent to $K_{[\lambda(j)] \mu} = 1$, generalizing a theorem of Berenshte\u{\i}n and Zelevinski\u{\i}.  We also show that the questions of whether $K_{[\lambda(j)] \mu} > 0$ or $K_{[\lambda(j)] \mu} = 1$ can be answered in polynomial time, expanding on a result of Narayanan.  Finally, we give an application to multiplicities in the degenerate Gel'fand-Graev representations of the finite general linear group, and we show that the problem of determining whether a given irreducible representation of the finite general linear group appears with nonzero multiplicity in a given degenerate Gel'fand-Graev representation, with their partition parameters as input, is $NP$-complete.
\end{abstract}

\section{Introduction} \label{intro}

Young tableaux and Kostka numbers, in their various forms, are of central importance in combinatorics and representation theory \cite{Sa90}.  In symmetric function theory, Kostka numbers appear in transition matrices between fundamental bases, and in representation theory they appear as multiplicity coefficients in numerous contexts.  An important question in both of these settings is whether a basic object appears with \emph{multiplicity one} in an expansion, since this uniqueness can be exploited to identify interesting properties of this object.  It is with this motivation that A. D. Berenshte\u{\i}n and A. V. Zelevinski\u{\i} \cite{BeZe90} solve the problem of, given a complex semisimple Lie algebra $\mathfrak{g}$ and an irreducible $\mathfrak{g}$-module $V_{\lambda}$ with highest weight $\lambda$, finding all weights $\omega$ such that the weight subspace $V_{\lambda}(\omega)$ of $V_{\lambda}$ has dimension $1$.  In the case that $\mathfrak{g}$ is of type $A$, the dimension of the weight subspace $V_{\lambda}(\omega)$ is exactly the Kostka number $K_{\lambda \omega}$, or the number of semistandard Young tableaux of shape $\lambda$ and content $\omega$ (which we will call the \emph{weight} of the tableau henceforth, and we will not refer to weights in the Lie algebra context again).  As an important corollary, Berenshte\u{\i}n and Zelevinski\u{\i} thus obtain conditions on partitions $\lambda$ and $\mu$ of a positive integer $n$ which are equivalent to the statement that $K_{\lambda \mu} = 1$.  The main purpose of this paper is to generalize this result in the case that we replace $\lambda$ with a multipartition $[\lambda(j)]$,  that is, a sequence of partitions with total size $n$, which we accomplish in Theorem \ref{MultiMultOne}.

The organization and results of this paper are as follows.  In Section \ref{prelims}, we give background material and the main definitions, starting with partitions and tableaux (Section \ref{parttab}), Kostka numbers and some classical results on symmetric functions and the symmetric group (Section \ref{Kostka}), and the definitions for multipartitions and multitableaux, including the basic results which parallel those for classical Kostka numbers (Section \ref{MultiDefs}).  In Section \ref{BZMult1}, we state the relevant result of Berenshte\u{\i}n and Zelevinski\u{\i} for Kostka numbers associated with partitions in Theorem \ref{BZ}.  In an effort to keep this paper as self-contained as possible, we give a complete proof of Theorem \ref{BZ} using tableau combinatorics.  There does not seem to be a combinatorial proof of this statement in the literature previously, and giving one here is useful to us in a few ways.  First, some of the ideas in the proof are used in the argument of the main result on multipartitions in Theorem \ref{MultiMultOne}, and second, through considering the details of the proof we are able to understand the computational complexity of answering the question of whether $K_{\lambda \mu} = 1$.  Narayanan \cite{Na06} showed that, given any partition $\lambda$ and composition $\omega$ of $n$, the question of whether $K_{\lambda \omega} >0$ can be answered in polynomial time.  We show in Corollary \ref{K1comp} that the same is true for the question of whether $K_{\lambda \omega} = 1$.

In Section \ref{wreath}, we turn to the representation theory of wreath products of finite groups to find motivation to study Kostka numbers associated with multipartitions.  The main result here is Theorem \ref{YoungWreath}, where we decompose wreath product permutation characters which are analogous to the symmetric group on a Young subgroup, and we find that multipartition Kostka numbers appear as multiplicities.  In Section \ref{Multi1}, we prove Theorem \ref{MultiMultOne}, in which we give conditions on a multipartition $[\lambda(j)]$ and a partition $\mu$ which are equivalent to the statement $K_{[\lambda(j)] \mu} = 1$.  We then show in Corollary \ref{Kmulticomp} that for a composition $\omega$, the questions of whether $K_{[\lambda(j)] \omega} > 0$ or $K_{[\lambda(j)] \omega} = 1$ can be answered in polynomial time.  Also, in Corollary \ref{J1Multi} we classify exactly which multipartitions $[\lambda(j)]$ are such that there exists a unique partition $\mu$ satisfying $K_{[\lambda(j)] \mu} = 1$, generalizing Corollary \ref{J1partitions} due to Z. Gates, B. Goldman, and the second-named author.

We lastly consider an application to the characters of the finite general linear group in Section \ref{GLnq}.  In particular, Zelevinsky decomposed the degenerate Gel'fand-Graev characters of the finite general linear group, with multiplicities being Kostka numbers associated with certain partition-valued functions.  We apply our main result Theorem \ref{MultiMultOne} to this setup to make some multiplicity one statements in Corollary \ref{ZelCor}.  Finally, in considering these Kostka numbers more generally, we show in Proposition \ref{NP} that the question of whether an irreducible character appears as a constituent of a degenerate Gel'fand-Graev character, with the partition data as input, is an $NP$-complete problem.  This somewhat surprising result seems to indicate new ideas are necessary in order to further understand these generalizations of Kostka numbers.
\\
\\
\noindent
{\bf Acknowledgements.  }  The authors thank Nick Loehr, Nat Thiem, Sami Assaf, Frans Schalekamp, Andreas Stathopoulos, and Arthur Gregory for helpful conversations.  The second-named author was supported in part by a grant from the Simons Foundation.

\section{Preliminaries} \label{prelims}

\subsection{Partitions and Tableaux} \label{parttab}

For any non-negative integer $n$, a \emph{partition} of $n$ is a finite sequence $\lambda = (\lambda_1, \lambda_2, \ldots, \lambda_l)$ of positive integers such that $\lambda_i \geq \lambda_{i+1}$ for $i < l$ and $\sum_{i=1}^l \lambda_i = n$, where $\lambda_i$ is the $i$th \emph{part} of $\lambda$.  The number of parts of $\lambda$ is the \emph{length} of $\lambda$, written $\ell(\lambda)$.  If $\lambda$ is a partition of $n$, we write $\lambda \vdash n$, and we say $\lambda$ has \emph{size} $n$, written $|\lambda| = n$.  If $n=0$, the only partition of $n$ is the \emph{empty partition}, written $(0)$.  We let $\mcP_n$ denote the set of all partitions of size $n$, and $\mcP$ the set of all partitions.  

We may represent a partition $\lambda = (\lambda_1, \lambda_2, \ldots, \lambda_l)$ by its \emph{Young diagram}, which has a row of $\lambda_i$ boxes for each part of $\lambda$, which are upper-left justified.  Notationally, we will often identify a partition $\lambda$ with its Young diagram.  For example, if $\lambda = (4, 3, 2, 2, 1)$, then $\lambda \vdash 12$ (so $|\lambda| = 12$), $\ell(\lambda) = 5$, and the Young diagram for $\lambda$ is
$$\yng(4,3,2,2,1).$$

Given $\lambda, \mu \in \mcP$, we say $\mu \subset \lambda$ if $\mu_i \leq \lambda_i$ for all $i$.  That is, $\mu \subset \lambda$ exactly when the Young diagram for $\mu$ fits inside of the Young diagram for $\lambda$.  If $\mu \subset \lambda$, we define $\lambda - \mu$ to be the result of removing the boxes of the Young diagram for $\mu$ from the Young diagram of $\lambda$, resulting in a \emph{skew diagram}.  For example, if $\lambda = (5, 4, 4, 1)$ and $\mu = (3, 2, 1)$, then $\mu \subset \lambda$, and the Young diagram for $\lambda$ with the Young diagram for $\mu$ inside of it marked by dots, and the skew diagram for $\lambda - \mu$, are
$$\young(\cdot\cdot\cdot\;\;,\cdot\cdot\;\;,\cdot\;\;\;,\;), \quad \young(:::\;\;,::\;\;,:\;\;\;,\;) \,.$$
The size of a skew diagram is defined to be the number of boxes it has, and above we have $|\lambda - \mu| = 8$.  A \emph{horizontal strip} is a skew diagram which has at most one box in each column, and a horizontal $m$-strip is a horizontal strip of size $m$.  For example, if $\lambda = (5,4,1)$ and $\mu = (4,2)$, then $\lambda - \mu$ is a horizontal $4$-strip with diagram
$$\young(::::\;,::\;\;,\;)\,.$$

There is a partial order on the set $\mcP_n$ called the \emph{dominance partial order} defined as follows.  Given two partitions $\lambda, \mu \in \mcP_n$, define $\lambda \unrhd \mu$ if for each $k \geq 1$, $\sum_{i \leq k} \lambda_i \geq \sum_{i \leq k} \mu_i$.  For example, if $\lambda = (3,2,1), \mu = (3,1,1,1)$, and $\nu = (2,2,2)$, then $\lambda \unrhd \mu$ and $\lambda \unrhd \nu$, but $\mu$ and $\nu$ are incomparable in the dominance partial order.

Given $\lambda \in \mcP$, a \emph{semistandard Young} (or \emph{column-strict})  \emph{tableau} of \emph{shape} $\lambda$ is a filling of the boxes of the Young diagram for $\lambda$ with positive integers, such that entries in each row weakly increase from left to right and entries in each column strictly increase from top to bottom.  If $T$ is a semistandard Young tableau of shape $\lambda$, the \emph{weight} of $T$, written $\wt(T)$, is the finite sequence $\omega = (\omega_1, \omega_2, \ldots, \omega_l)$, where $\omega_i$ is the number of times $i$ is an entry in $T$, and $l$ is the largest entry in $T$.  Note that $\sum_{i=1}^l \omega_i = |\lambda|$, so if $\lambda \vdash n$, then $\wt(T)$ is a \emph{composition} of $n$ (where we allow entries to be $0$, and the length $l$ is the position of the last positive part).  We will mainly concern ourselves with the case that the weight of a tableau is also a partition, so that $\omega_i \geq \omega_{i+1}$ for each $i$, and we will typically denote the weight by $\mu$ in this case.  An example of a semistandard Young tableau of shape $\lambda = (4,2,2,1)$ and weight $\mu = (3,3,2,1)$ is
$$\young(1112,22,33,4).$$
We may also describe a semistandard Young tableau of shape $\lambda$ and weight $\omega$ to be a sequence
$$(0) = \lambda^{(0)} \subset \lambda^{(1)} \subset \cdots \subset \lambda^{(l)} = \lambda$$
of nested partitions, such that each $\lambda^{(i)} - \lambda^{(i-1)}$ is a (possibly empty) horizontal strip, and $\omega_i = |\lambda^{(i)} - \lambda^{(i-1)}|$.  The idea is that the horizontal strip $\lambda^{(i)} - \lambda^{(i-1)}$ is exactly the set of boxes with entry $i >0$ in the corresponding semistandard Young tableau, and we take $\lambda^{(i)} = \lambda^{(i-1)}$ when there are no $i$ entries.

\subsection{Kostka numbers} \label{Kostka}

Let $\lambda \in \mcP_n$ and let $\omega = (\omega_1, \omega_2, \ldots, \omega_l)$ be a composition of $n$.  The number of semistandard Young tableau of shape $\lambda$ and weight $\omega$ is the \emph{Kostka number}, denoted $K_{\lambda \omega}$.  Kostka numbers play an important role in algebraic combinatorics and representation theory.  A crucial example is in symmetric function theory (see \cite[I.2]{Ma95} for an introduction).  For $m \geq 0$, let $h_m$ be the complete symmetric function, and if $\mu, \lambda \vdash n$, where $\mu = (\mu_1, \mu_2, \ldots, \mu_l)$, let $h_{\mu} = h_{\mu_1} \cdots h_{\mu_l}$ and let $s_{\lambda}$ denote the Schur symmetric function corresponding to $\lambda$.  Then the Kostka numbers make up the entries of the transition matrix between these two bases of symmetric functions, in that we have
\begin{equation} \label{KostkaSymm}
h_{\mu} = \sum_{\lambda \vdash n} K_{\lambda \mu} s_{\lambda}.
\end{equation}
The proof of \eqref{KostkaSymm} uses Pieri's formula \cite[I.5.16]{Ma95}, which states that for any $\lambda \in \mcP$ and any $m \geq 0$, 
$$h_m s_{\lambda} = \sum_{\nu} s_{\nu},$$
where the sum is taken over all partitions $\nu$ such that $\nu - \lambda$ is a horizontal $m$-strip.  Recalling that $h_n = s_{(n)}$, \eqref{KostkaSymm} is obtained by inductively applying Pieri's formula and observing that $s_{\lambda}$ appears in the expansion of $h_{\mu}$ in as many ways as we can build $\lambda$ by choosing horizontal strips of sizes $\mu_1$, $\mu_2$, and so on.  By the definition of semistandard Young tableaux through horizontal strips, this multiplicity is exactly $K_{\lambda \mu}$.

In the above argument, note that changing the order of the product $h_{\mu} = h_{\mu_1} \cdots h_{\mu_l}$ does not change the expansion, and the same Pieri rule argument may be applied no matter what the order of these factors, and in that case the $\mu$ in \eqref{KostkaSymm} is replaced by the composition resulting in permuting the parts of $\mu$ (or inserting $0$'s).  That is to say, it is a result of the proof of \eqref{KostkaSymm} that for a composition $\omega$ of $n$ and $\lambda \in \mcP_n$, the Kostka number $K_{\lambda \omega}$ is invariant under permutation of the parts of $\omega$, see also \cite[Section 4.3, Proposition 2]{Fu97}.  This is precisely why we may restrict our attention to the case of partitions being the weights of semistandard Young tableaux.  

It is also known that for $\lambda, \mu \in \mcP_n$, we have $K_{\lambda \mu} > 0$ if and only if $\lambda \unrhd \mu$.  The fact that $K_{\lambda \mu} > 0$ implies $\lambda \unrhd \mu$ follows from a quick argument, and we give the more general proof for multipartitions in Lemma \ref{KosSeqNonzero}.  The converse statement is more subtle.  We give a constructive proof here (coming out of a discussion with Nick Loehr) since it is not typically found in the literature, and we need the construction explicitly in Section \ref{BZMult1}.  Suppose $\lambda \unrhd \mu$, which implies $\ell(\lambda) \leq \ell(\mu)$, and let $\mu= (\mu_1, \ldots, \mu_l)$ with $\ell(\mu) = l$ and write $\lambda = (\lambda_1, \ldots, \lambda_l)$ with $\ell(\lambda) \leq l$ so that some $\lambda_i$ may be $0$.  We construct a semistandard Young tableau $T$ of shape $\lambda$ and weight $\mu$ as follows.  We fill horizontal strips greedily from the bottom of $\lambda$ to the top, by filling in the longest columns first, where rows are filled from right to left.  That is, we begin by filling in $\mu_l$ entries of $l$, starting from right to left in the last row of $\lambda$, and if that row is filled, we continue the horizontal strip on the next available row, and continue until using all $l$'s.  We then fill in the next horizontal strip up in the same way.  For example, if $\lambda = (12, 4, 2, 2)$ and $\mu = (5, 5, 5, 5)$, then the partial tableau looks like the following after filling in the $\mu_4 = 5$ entries of $4$:
$$\young(\;\;\;\;\;\;\;\;\;\;\;4,\;\;44,\;\;,44) \,.$$
Let $j$ be maximal such that $\lambda_j \geq \mu_l$, so that row $j$ is the upper-most row which contains an $l$ entry, and let $r_i$, $j \leq i \leq l$, be the number of $l$ entries which occur in row $i$ of the partial tableau after this step, so that $\sum_{i=j}^l r_i = \mu_l$.  In the above example, $j=1, r_1 = 1, r_2 = 2, r_3 = 0$, and $r_4 = 2$.  Let $\mu^* = (\mu_1, \ldots, \mu_{l-1})$, and define $\lambda^*$ to have parts $\lambda^*_i = \lambda_i$ if $i < j$, and $\lambda^*_i = \lambda_i - r_i$ if $i \geq j$.  To show that we can continue these steps to construct a semistandard Young tableau of shape $\lambda$ and weight $\mu$, then it is enough by induction on $n$ to prove that $\lambda^* \unrhd \mu^*$.  It is immediate that for $k < j$, we have $\sum_{i=1}^k \lambda^*_i \geq \sum_{i=1}^k \mu^*_i$, since $\lambda \unrhd \mu$.  In the case that $\mu_l \leq \lambda_l$, then $j = l$ and $r_l = \mu_l$, and $\lambda^* \unrhd \mu^*$ follows immediately.  So we assume now that $\mu_l - \lambda_l = d > 0$.  Consider $k$ such that $j \leq k < l$, and we must show $\sum_{i=1}^k \lambda_i^* \geq \sum_{i=1}^k \mu_i^*$, where $\mu_i^* = \mu_i$ when $i \leq k < l$.  Note that whenever $k < i < l$, then $\lambda_i < \mu_l$ since $i > j$, and $\mu_i \leq \mu_l$ since $\mu$ is a partition.  Thus $\lambda_i -\mu_i \leq 0$, while $\lambda_l - \mu_l = -d$, so that $\sum_{i = k+1}^l \lambda_i - \sum_{i=k+1}^l \mu_i \leq -d$.  Since $\sum_{i=1}^l \lambda_i - \sum_{i=1}^l \mu_i = 0$, then we obtain
\begin{equation} \label{tabsum}
\sum_{i = 1}^{k} \lambda_i - \sum_{i=1}^k \mu_i \geq d.
\end{equation}
Since $\lambda_l < \mu_l$, then we have $r_l = \lambda_l$, and so $\mu_l - r_l = \sum_{i=j}^{l-1} r_i = \mu_l - \lambda_l = d$.  So $\sum_{i=j}^k r_i \leq d$.  From this fact and \eqref{tabsum}, we now have
$$ \sum_{i=1}^k \lambda^*_i - \sum_{i=1}^k \mu^*_i = \sum_{i=1}^k \lambda_i - \sum_{i = 1}^k \mu_i - \sum_{i = j}^k r_i \geq d - d = 0.$$
Thus $\lambda^* \unrhd \mu^*$ as claimed, and we can construct a semistandard Young tableau of shape $\lambda$ and weight $\mu$ in this way whenever $\lambda \unrhd \mu$.

Now consider the symmetric group $S_n$ of permutations of $\{1, 2, \ldots, n\}$, and the irreducible complex representations of $S_n$.  These representations are parameterized by $\mcP_n$, and we denote the irreducible representation of $S_n$ parameterized by $\lambda \vdash n$ by $\pi^{\lambda}$, and its character by $\chi^{\lambda}$.  We adapt the convention of \cite[Lecture 4]{FuHa91} and \cite[I.8]{Ma95} that $\chi^{(n)} = \bone$, the trivial character of $S_n$, and $\chi^{(1,1,\ldots,1)}$ is the sign character of $S_n$.  Given $\mu = (\mu_1, \ldots, \mu_l) \vdash n$, let $S_{\mu}$ be the \emph{Young subgroup}
$$S_{\mu} = S_{\mu_1} \times S_{\mu_2} \times \cdots \times S_{\mu_l},$$
where $S_{\mu}$ is embedded in $S_n$ such that $S_{\mu_1}$ permutes $1$ through $\mu_1$, $S_{\mu_2}$ permutes $\mu_1 + 1$ through $\mu_2$, and so on.  If we consider the permutation representation of $S_n$ on $S_{\mu}$, we obtain the decomposition
\begin{equation} \label{KostkaPerm}
\Ind_{S_{\mu}}^{S_n}(\bone) = \bigoplus_{\lambda \unrhd \mu} K_{\lambda \mu} \pi^{\lambda},
\end{equation}
known as Young's rule, see \cite[Corollary 4.39]{FuHa91} for example.  For any partition $\mu$, one may see that $K_{\mu\mu} = 1$, so that $\pi^{\mu}$ appears with multiplicity 1 in the decomposition \eqref{KostkaPerm}. Two natural question arise.  Which other $\pi^{\lambda}$ occur with multiplicity 1 in the representation $\Ind_{S_{\mu}}^{S_n}(\bone)$, and when is it the case that $\pi^{\lambda}$ appears with multiplicity 1 in $\Ind_{S_{\mu}}^{S_n}(\bone)$ only for $\mu=\lambda$?  We give answers to these two questions in Section \ref{BZMult1}.

\subsection{Multipartitions and multitableaux} \label{MultiDefs}

If $n \geq 0$ and $r \geq 1$ are integers, then an \emph{$r$-multipartition} of $n$ is a sequence of $r$ partitions such that the sum of the sizes of the $r$ partitions is $n$.  We denote an $r$-multipartition of $n$ by $[\lambda(j)] = [\lambda(1), \lambda(2), \ldots, \lambda(r)]$, where $[\lambda(j)]$ has size $|[\lambda(j)]| = \sum_{j=1}^r |\lambda(j)| = n$.  Let $\mcP_n[r]$ denote the set of all $r$-multipartitions of $n$, and let $\mcP[r]$ denote the set of all $r$-multipartitions.  Given $[\lambda(j)] \in \mcP_n[r]$, we define the partition 
\begin{equation} \label{tildedef}
\tlam \in \mcP_n \quad \text{which has parts} \quad \tlam_i = \sum_{j=1}^r \lambda(j)_i,
\end{equation}
so that $\ell(\tlam) = \mathrm{max} \{ \ell(\lambda(j)) \, \mid \, 1 \leq j \leq r \}$.  One may also view $\tlam$ as the partition obtained by taking all columns of all $\lambda(j)$, and arranging them from longest to smallest.  For example, suppose $[\lambda(j)] \in \mcP_{12}[3]$, with $\lambda(1) = (2, 1, 1), \lambda(2) = (2, 2), \lambda(3) = (4)$.  Then $\tlam = (8, 3, 1)$, with Young diagrams
$$ \lambda(1) = \young(\;\;,\;,\;) \; , \quad \lambda(2) = \young(\;\;,\;\;)\; , \quad \lambda(3) = \young(\;\;\;\;)\; , \quad \text{and} \quad \tlam = \young(\;\;\;\;\;\;\;\;,\;\;\;,\;).$$
Given $[\lambda(j)] \in \mcP_n[r]$ and a composition $\omega$ of $n$, a \emph{semistandard Young $r$-multitableau} of \emph{shape} $[\lambda(j)]$ and \emph{weight} $\omega$ is a sequence $[T(j)]$ of $r$ semistandard Young tableaux, where each $T(j)$ has shape $\lambda(j)$, and if $\omega(j)$ is the weight of $\lambda(j)$, then $\omega$ has parts given by $\omega_i = \sum_{j=1}^r \omega(j)_i$.  As was the case with single tableau, we will primarily deal with $r$-multitableau which have total weight given by a partition $\mu$.  

For example, if $[\lambda(j)] \in \mcP_{12}[3]$ is the same as in the previous example, one semistandard Young $3$-multitableau of weight $\mu = (4,3,3,1,1)$ is given by 
$$ T(1) = \young(13,3,4)\; , \quad T(1) = \young(12,23)\;, \quad T(3) = \young(1125) \,.$$
Note that the weight $\mu$ of $[T(j)]$ is a partition, while none of the weights of the individual tableaux $T(j)$ are partitions.

Given $[\lambda(j)] \in \mcP_n[r]$ and any composition $\omega$ of $n$, we define the \emph{multipartition Kostka number} $K_{[\lambda(j)] \omega}$ to be the total number of semistandard Young $r$-multitableaux of shape $[\lambda(j)]$ and weight $\omega$.  The following relates the Kostka number for an $r$-multipartition to the Kostka number for the individual partitions.

\begin{proposition} \label{compdecomp} For any composition $\upsilon$ of length $l$ and size $n$, and any $[\lambda(j)] \in \mcP_n[r]$, we have
$$ K_{[\lambda(j)] \upsilon} = \sum_{\sum_j \omega(j) = \upsilon \atop{ |\omega(j)| = |\lambda(j)|}} \prod_{j=1}^r K_{\lambda(j) \omega(j)},$$
where the sum is over all possible ways to choose $r$ $l$-tuples $\omega(j)$, $1 \leq j \leq r$, of non-negative integers, where the sum of the $\omega(j)$ coordinate-wise is $\mu$, and the size of $\omega(j)$ taken as a composition is the same as the size of $\lambda(j)$.
\end{proposition}
\begin{proof} We may count possible semistandard Young $r$-multitableaux of weight $\upsilon$ and shape $[\lambda(j)]$ as follows.  For each $j$, we choose the weight $\omega(j)$ of the tableau of shape $\lambda(j)$.  We must choose these $\omega(j)$ so that the coordinate-wise sum is exactly the weight $\mu$ for the whole $r$-multitableau.  Now $K_{\lambda(j) \omega(j)}$ is the total number of semistandard Young tableaux of shape $\lambda(j)$ and weight $\omega(j)$, so that once we fix the individual weights $\omega(j)$, the total number of $r$-multitableaux with this prescription is the product $\prod_{j=1}^r K_{\lambda(j) \omega(j)}$.  The formula follows.  
\end{proof}

Next we observe that, like in the case for partitions, the Kostka number for a multipartition is invariant under permuting the parts of the weight.  This is precisely why we may restrict our attention to the weight being a partition.

\begin{corollary} \label{PermInvar} For any $[\lambda(j)] \in \mcP_n[r]$, and any composition $\upsilon$ of size $n$, let $\mu$ be the unique partition obtained by permuting the parts of $\upsilon$.  Then $K_{[\lambda(j)] \upsilon} = K_{[\lambda(j)] \mu}$.
\end{corollary}
\begin{proof}  If $\upsilon$ has length $l$, let $\sigma \in S_l$ be the permutation which, when applied to the parts of $\upsilon$, gives $\mu$.  If we use Proposition \ref{compdecomp}, and we apply $\sigma$ to every $l$-tuple $\omega(j)$ in the sum when computing $K_{[\lambda(j)] \upsilon}$, then we obtain every $l$-tuple $\omega'(j)$ which would appear in the sum when computing $K_{[\lambda(j)] \mu}$.  Since Kostka numbers for partitions are invariant under permutation of the parts of the weight, then we always have $K_{\lambda(j) \omega(j)} = K_{\lambda(j) \omega'(j)}$ when $\omega'(j)$ is the result of permuting the parts of $\omega(j)$ by $\sigma$.  It follows that $K_{[\lambda(j)] \upsilon} = K_{[\lambda(j)] \mu}$.
\end{proof}

The following result gives a precise condition for when the Kostka number for an $r$-multipartition is nonzero, and reduces to the case of partitions.

\begin{lemma} \label{KosSeqNonzero} Suppose $[\lambda(j)] \in \mcP_n[r]$ and $\mu \in \mcP_n$.  Then $K_{[\lambda(j)] \mu} > 0$ if and only if $\tlam \unrhd \mu$.
\end{lemma}
\begin{proof} First suppose that $\tlam$ does not dominate $\mu$ (that is, suppose $\tlam \unrhd \mu$ does not hold), but there does exist a semistandard Young $r$-multitableau $[T(j)]$ of shape $[\lambda(j)]$ and weight $\mu$.  Then for some $m \geq 1$, we have $\sum_{i \leq m} \tlam_i < \sum_{i \leq m} \mu_i$, which means
\begin{equation} \label{wrongway}
\sum_{j=1}^r \sum_{i \leq m} \lambda(j)_i < \sum_{i \leq m} \mu_i.
\end{equation}
Since the entries in each column of each $T(j)$ strictly increase, then all entries $1$ through $m$ must appear in the first $m$ rows of each $T(j)$.  However, the inequality \eqref{wrongway} says that there are more entries $1$ through $m$ than there are boxes in the first $m$ rows of $[T(j)]$, giving a contradiction.

Now assume $\tlam \unrhd \mu$.  Then $K_{\tlam \mu} >0$, and so there is a semistandard Young tableau $T$ of shape $\tlam$ and weight $\mu$.   Given such a $T$, we may construct a semistandard Young $r$-multitableau $[T(j)]$ of shape $[\lambda(j)]$ and weight $\mu$ as follows.  Each column of $\tlam$ is the same length of some column of some $\lambda(j)$, and we may take each column of $T$, and make it a column of some $T(j)$.  If some column in $T$ is to the left of another column in $T$, we need only make sure that if these columns are in the same $T(j)$, that the first is still to the left of the second.  This guarantees that each $T(j)$ is a semistandard Young tableau, since the entries of each column strictly increase going down, as they did in $T$, and the entries in each row of $T(j)$ weakly increase to the right, since each column is arranged in an order so that weak increasing is preserved from $T$.  Thus any $[T(j)]$ created in this way is a semistandard Young $r$-multitableau of shape $[\lambda(j)]$ and weight $\mu$, so that $K_{[\lambda(j)] \mu} > 0$.
\end{proof}

The construction of the $r$-multitableau $[T(j)]$ from the tableau $T$ in the end of the proof above will be used again, and so we give an example now.  Let $\mu = (7, 6, 4, 3, 1, 1)$, and define $[\lambda(j)]$ by $\lambda(1) = (4, 4, 3)$, $\lambda(2) = (3, 3, 1)$, and $\lambda(3) = (3, 1)$.  Then $\tlam = (10, 8, 4)$.  One semistandard Young tableau $T$ of shape $\tlam$ and weight $\mu$ is
$$ T = \young(1111111222,22233334,4456).$$
By choosing columns of $T$ to go in different positions of an $r$-multitableau, while still preserving the order of entries, we can construct (at least) two different semistandard $r$-multitableau, $[T(j)]$ and $[T'(j)]$, of shape $[\lambda(j)]$ and weight $\mu$, which we may define as
$$T(1) = \young(1112,2234,446) \;, \quad T(2)=\young(111,233,5) \;, \quad T(3) = \young(122,3) \; ,$$
and 
$$T'(1) = \young(1111,2233,456) \;, \quad T'(2)=\young(111,233,4) \;, \quad T'(3) = \young(222,4)\,.$$

\section{Multiplicity One for Partitions} \label{BZMult1}

Berenshte\u{\i}n and Zelevinski\u{\i} \cite{BeZe90} give precise conditions on partitions $\lambda, \mu$, which are equivalent to the statement that $K_{\lambda \mu} = 1$.  We note that there is a very slight typographical error in the subscripts of \cite[Theorem 1.5]{BeZe90} which we correct below.

\begin{theorem}[Berenshte\u{\i}n and Zelevinski\u{\i}] \label{BZ} Let $\lambda, \mu \in \mcP_n$, and suppose $\ell(\mu) = l$.  Then $K_{\lambda \mu} = 1$ if and only if there exists a choice of indices $0 = i_0 < i_1 < \cdots < i_t = l$ such that, for $k = 1,\ldots,t$, the partitions 
$$ \lambda^k = (\lambda_{i_{k-1} + 1}, \lambda_{i_{k-1} + 2}, \ldots, \lambda_{i_k}) \quad \text{and} \quad \mu^k = (\mu_{i_{k-1} + 1}, \mu_{i_{k-1} + 2}, \ldots, \mu_{i_k}),$$
where we define $\lambda_i = 0$ if $i > \ell(\lambda)$, satisfy the following:
\begin{enumerate}
\item[(1)] $\lambda^k \unrhd \mu^k$, and
\item[(2)] either $\lambda_{i_{k-1} + 1} = \lambda_{i_{k-1} + 2}= \cdots= \lambda_{i_k - 1}$ or $\lambda_{i_{k-1}+1} > \lambda_{i_{k-1} + 2}= \lambda_{i_{k-1} + 3}= \cdots= \lambda_{i_k}$.
\end{enumerate}
\end{theorem}

As our main result in Section \ref{Multi1} is a generalization of and heavily dependent on Theorem \ref{BZ}, we give a tableau-theoretic proof of Theorem \ref{BZ} now.  We begin with the following special cases.

\begin{lemma} \label{ColLemma} Let $\lambda, \mu \in \mcP_n$ with $\ell(\mu) = l$, $\ell(\lambda) \leq l$, $\lambda \unrhd \mu$, and $\lambda = (\lambda_1, \lambda_2, \ldots, \lambda_l)$, with $\lambda_i = 0$ if $i > \ell(\lambda)$.  Suppose that either:
\begin{enumerate}
\item[(1)] $\lambda_1 = \lambda_2 = \cdots = \lambda_l$, 
\item[(2)] $\lambda_1 > \lambda_2 = \lambda_3 = \cdots = \lambda_l$, or
\item[(3)] $\lambda_1 = \lambda_2 = \cdots = \lambda_{l-1} > \lambda_l$.
\end{enumerate}
Then $K_{\lambda \mu} = 1$.  In case (1), or in case (2) when $\lambda_1 - \lambda_2 = 1$, or in case (3) when $\lambda_{l-1} - \lambda_l = 1$, we must also have $\mu = \lambda$.
\end{lemma}
\begin{proof}   Since $\lambda \unrhd \mu$, there is at least one semistandard Young tableau of shape $\lambda$ and weight $\mu$.  Let $T$ be such a tableau.

In case (1), we have $\ell(\lambda) = \ell(\mu)$, and the only case for $T$ is that every column of $T$ has entries $1$ through $l$ in sequential order, since entries must strictly increase from top to bottom in columns.  Thus $K_{\lambda \mu} = 1$ and $\mu = \lambda$.

In case (2), if we have $\lambda_2 = 0$, then $\lambda = (\lambda_1)$, and we must have $K_{\lambda \mu} = 1$ since any semistandard Young tableau $T$ of shape a single row must have entries in weakly increasing order.  So assume $\lambda_2 > 0$, in which case $\ell(\lambda) = \ell(\mu)$.  Since $\ell(\mu) = l$ is the number of rows of $T$, then the first $\lambda_2$ columns of $T$ must have exactly the entries $1$ through $l$ in that order.  The remaining $\lambda_1 - \lambda_2$ entries must be in weakly increasing order in row 1 of $T$.  For example, if $\lambda = (6,3,3)$ and $\mu = (5, 4, 3)$, then the first $3$ columns of the tableau must have entries $1, 2, 3$, and the remaining entries must be in row 1 in a fixed order: 
$$\young(111\;\;\;,222,333) \quad \mapsto \quad \young(111112,222,333).$$
There is thus only one such $T$, and $K_{\lambda \mu} = 1$.  If $\lambda_1 - \lambda_2 = 1$, then the only entry that can appear in the right-most entry of the first row of $T$ is $1$, in order for the weight $\mu$ to be a partition.  Thus $\mu = \lambda$ in this case.

In case (3), first consider the case that $\lambda_l = 0$, so that $\ell(\lambda) = l -1$.  In this case, each column of $T$ is missing exactly one of the entries $1$ through $l$.  Since the entries of each column strictly increase from top to bottom, knowing the missing entry is the same as knowing the column entries.  Since also row entries weakly increase from left to right, then given columns of length $l-1$ with entries from $1$ through $l$, the columns can only be arranged in one way to form $T$, which is in such a way that the missing entries from each column weakly decrease from left to right.  In other words, $T$ is uniquely determined from $\lambda$ and $\mu$, and $K_{\lambda \mu} = 1$.  For example, if $\lambda = (5, 5, 5, 5)$ and $\mu = (5, 4, 4, 4, 3)$, then every column has a $1$, one column each is missing a $2, 3$, and $4$, and one column is missing a $5$.  The only such $T$ with this property is
$$ \young(11111,22223,33344,44555) \, ,$$
where the missing entries from each column, from left to right, are $5, 5, 4, 3$, and $2$.  Note that we cannot have $\lambda_{l-1} - \lambda_{l} = 1$ and $\ell(\mu) = l$ in this case.  Now consider the case that $\lambda_l > 0$.  The first $\lambda_l$ columns of $\lambda$ are then of length $l$, and so these columns of $T$ must have the entries $1$ through $l$ from top to bottom.  If we remove these columns from $\lambda$ and these entries from $\mu$, we are left with a partition of length $l-1$ with all equal parts, and now the case just covered (case (3) with $\lambda_l = 0$) applies.  Thus $K_{\lambda \mu} = 1$.  If $\lambda_{l-1} - \lambda_l = 1$, then again the entries of the first $\lambda_l$ columns of $T$ must be $1$ through $l$ from top to bottom, and the only choice for the last column of $T$ to have the weight $\mu$ a partition, is to have the entries $1$ through $l-1$ from top to bottom.  This implies $\mu = \lambda$.
\end{proof}

We will also need the following result.

\begin{lemma} \label{Rowilastrow}  Let $\lambda, \mu \in \mcP_n$ such that $K_{\lambda \mu} = 1$ and $\ell(\lambda) = h \leq \ell(\mu) = l$.  Then the unique semistandard Young tableau $T$ of shape $\lambda$ and weight $\mu$ has the following properties.
\begin{enumerate}
\item[(1)] The first entry in each row $i$ of $T$ is $i$ for $i = 1, \ldots, h$.
\item[(2)] If $h=l$, then every entry of row $h$ of $T$ is $h$.  If $h < l$, then every entry greater than $h$ in $T$ is in row $h$ of $T$, and further $\sum_{j=h+1}^l \mu_j < \lambda_h$.
\end{enumerate}
\end{lemma}
\begin{proof}  For (1), by way of contradiction suppose that $i$ is minimal such that the first entry of row $i$ of $T$ is not $i$.  If the first entry of row $i$ is $j > i$, then this $j$ entry does not have a $j-1$ entry directly above it.  Since $\mu_{j-1} \geq \mu_j$, then $T$ must have a $j-1$ entry with no $j$ entry directly below it.  In particular, such an entry has either no square below it, or has an entry greater than $j$ below it.  This implies that we may switch this $j-1$ entry with the first $j$ entry in row $i$, resulting in another distinct semistandard Young tableau of shape $\lambda$ and weight $\mu$, contradicting the fact that $K_{\lambda \mu} = 1$.

For statement (2), we consider the construction of $T$ by greedily filling horizontal strips from bottom to top, as described in Section \ref{Kostka}.  Since $K_{\lambda \mu} = 1$, then this construction must yield the unique tableau $T$.  By the construction, we begin by filling row $h$ of $T$ from right to left with the $l$ entries, and moving to entries in shorter columns if we fill up row $h$, and then continuing with smaller entries on the next horizontal strip up.  We note that by part (1), $T$ must have the property that row $h$ has first entry $h$, which means that when constructing $T$ in this way, if $h = l$, then every entry of row $h$ must be $h$, and if $h < l$ then we never fill up row $h$ with entries greater than $h$.  This implies that $T$ has the property that all entries greater than $h$ must be in row $h$, and these do not occupy all of row $h$.  In other words, $\mu_{h+1} + \cdots + \mu_l < \lambda_h$.
\end{proof}

We remark that the proof of property (1) in Lemma \ref{Rowilastrow} also implies that whenever $\lambda \unrhd \mu$, there always exists a semistandard Young tableau $T$ of shape $\lambda$ and weight $\mu$ such that each row $i$ of $T$ has first entry $i$, for $i = 1, \ldots, h$, although we will not need this statement here.  

\begin{proof}[Proof of Theorem \ref{BZ}]  First we assume that $\lambda, \mu \in \mcP_n$ satisfy the conditions listed in Theorem \ref{BZ}.  Observe that $\sum_{k=1}^t |\mu^k| = |\mu| = n$, and since $|\lambda^k| = |\mu^k|$ for $k = 1, \ldots, t$, and $|\lambda| = n$, then every part of $\lambda$ must be a part of some $\lambda^k$, and so $\ell(\mu) = l  \geq \ell(\lambda)$.  Also, since $\lambda^k \unrhd \mu^k$ for each $k$, it follows that $\lambda \unrhd \mu$, and so $K_{\lambda \mu} \geq 1$.

Consider a semistandard Young tableau $T$ of shape $\lambda$ and weight $\mu$.  Since entries strictly increase down columns of $\lambda$, any entry in the $i$th row of $T$ must be at least $i$.  Consider rows $1$ through $i_1$ of $T$.  Since any row past $i_1$ must have entries at least $i_1 + 1$, then all $|\mu^1|$ entries of value at most $i_1$ must occur in the first $i_1$ rows of $T$.  Since $|\lambda^1| = |\mu^1|$, it follows that the \emph{only} entries in the first $i_1$ rows of $T$ can be $1$ through $i_1$.  By induction, the only entries in rows $i_{k-1} + 1$ through $i_k$ can be the values $i_{k-1} + 1$ through $i_k$.  It follows that we may consider tableaux of shapes $\lambda^k$ and weights $\mu^k$ independently, and that $K_{\lambda \mu} = \prod_{k=1}^t K_{\lambda^k \mu^k}$.  It is therefore enough to show that for $\lambda, \mu \in \mcP_n$, with $\ell(\mu) = l$, if $\lambda \unrhd \mu$ and either $\lambda_1 = \lambda_2 = \cdots = \lambda_{l-1}$ or $\lambda_1 > \lambda_2 = \lambda_3 = \cdots = \lambda_l$ (where $\lambda_i = 0$ if $i > \ell(\lambda)$), then $K_{\lambda \mu} = 1$.  This is exactly what is proved in Lemma \ref{ColLemma}.  

We now assume $K_{\lambda \mu} = 1$, and we prove there always exists a choice of indices $0 = i_0 < i_1 < i_2 < \ldots < i_t = l$ with the desired properties.  Throughout, we take $T$ to be the unique semistandard Young tableau of shape $\lambda$ and weight $\mu$.  We prove the statement by induction on $n$, where in the case $n=1$, we have $\lambda = (1) = \mu$, and the result is immediate.  So we assume $n > 1$, and that the statement holds for all positive integers less than $n$.  We note that the statement quickly follows for $\lambda = (n)$ or $\lambda = (1, 1, \ldots, 1)$, and so we may assume that $\lambda$ is not of these forms.  We consider four scenarios in the induction.  

Let $\ell(\lambda) = h \leq \ell(\mu) = l$.  We first consider the case that $\mu_{h+1} < \mu_h$ (which includes the case that $h=l$) and $\mu_h > 1$.  Define new partitions $\lambda^*$ and $\mu^*$ by $\lambda^*_i = \lambda_i - 1$ and $\mu^*_i = \mu_i - 1$ for $1 \leq i \leq h$, and $\mu^*_i = \mu_i$ for $i > h$.  The assumption that $\mu_{h+1} < \mu_h$ guarantees that $\mu^*$ is indeed a partition, and we have $\lambda^*, \mu^* \in \mcP_{n-h}$.  Since $\lambda \unrhd \mu$, we also have $\lambda^* \unrhd \mu^*$.  Consider any semistandard Young tableau $T^*$ of shape $\lambda^*$ and weight $\mu^*$.  Given $T^*$, add a box with an $i$ entry at the beginning of row $i$ for $i = 1, \ldots, h$ (noting that some of these rows of $T^*$ might be initially length $0$).  Since the least possible entry in row $i$ of a semistandard Young tableau is $i$, then this construction yields a semistandard Young tableau of shape $\lambda$ and weight $\mu$, which then must be $T$ since $K_{\lambda \mu} = 1$, and note that $T$ has the property that the first entry of each row $i$ is $i$ by Lemma \ref{Rowilastrow}.  Consequently, $T^*$ must be the unique semistandard Young tableau of shape $\lambda^*$ and weight $\mu^*$, since a distinct such tableau would yield a distinct tableau of shape $\lambda$ and weight $\mu$.  Thus, $K_{\lambda^* \mu^*} = 1$, and we may apply the induction hypothesis to $\lambda^*$ and $\mu^*$.  Let $0 = i_0^* < i_1^*< \cdots < i_t^* = \ell(\mu^*)$ be the desired indices for $\lambda^*$ and $\mu^*$, which exist by the induction hypothesis, with the accompanying subpartitions $\lambda^{*k}$ and $\mu^{*k}$.  Note that $\ell(\mu^*) = \ell(\mu)$ since $\mu_h > 1$.     Now let $i_k = i_k^*$ for $k = 0, 1, \ldots, t$, and consider the corresponding subpartitions $\lambda^k$ and $\mu^k$.  It immediately follows that since $\lambda^{*k} \unrhd \mu^{*k}$, then $\lambda^k \unrhd \mu^k$ for $k = 1, \ldots, t$.  Also, since either $\lambda^*_{i^*_{k-1} + 1} = \lambda^*_{i^*_{k-1} +2} = \cdots = \lambda^*_{i^*_k -1}$ or $\lambda^*_{i^*_{k-1} + 1} > \lambda^*_{i^*_{k-1} +2} = \lambda^*_{i^*_{k-1}+3} = \cdots = \lambda^*_{i^*_k}$, then either $\lambda_{i_{k-1} + 1} = \lambda_{i_{k-1} +2} = \cdots = \lambda_{i_k -1}$ or $\lambda_{i_{k-1} + 1} > \lambda_{i_{k-1} +2} = \lambda_{i_{k-1}+3} = \cdots = \lambda_{i_k}$, respectively, where the $k<t$ case follows immediately.  The $k=t$ case follows from the assumption that $\mu_h > 1$, because if $\lambda_h = 1$ then $\ell(\lambda) = \ell(\mu)$ (since the last row of $T$ would have no room for entries greater than $h$), and if $\lambda_h > 1$ then $\ell(\lambda) = \ell(\lambda^*)$, and these statements imply the desired subpartition properties are preserved.

If $\mu_{h+1} < \mu_h$ and $\mu_h =1$, then $\mu_{h+1} = 0$ and $h = l$, and $\lambda_h = 1$, which follows from Lemma \ref{Rowilastrow} since all entries in row $h$ of $T$ must be $h$.  Now let $\lambda^* = (\lambda_1, \ldots, \lambda_{l-1})$ and $\mu^* = (\mu_1, \ldots, \mu_{l-1})$, and we must have $K_{\lambda^* \mu^*} = 1$ since we may add a single box with entry $l$ at the bottom of any semistandard Young tableau of shape $\lambda^*$ and weight $\mu^*$ to obtain a tableau which must be $T$.  We take the indices obtained by applying the induction hypothesis to $\lambda^*, \mu^* \in \mcP_{n-1}$, and we add to it the index $i_t = l$ to obtain indices for $\lambda$ and $\mu$.  It immediately follows that the resulting subpartitions have the desired properties.

Now consider the case that $\mu_{h+1} = \mu_h$ (so $\ell(\lambda) = h < \ell(\mu) = l$) and either $l > h+1$ or $\lambda_h = \lambda_{h-1}$.  By Lemma \ref{Rowilastrow}, all entries greater than $h$ must be in row $h$ of $T$.  Now define $\lambda^*$ and $\mu^*$ by $\lambda^*_i = \lambda_i$ if $i < h$, and $\lambda^*_h = \lambda_h -\mu_l$, and $\mu^*_i = \mu_i$ if $i < l$, and $\mu^*_i = 0$ if $i =l$.  Note that $\lambda^*_h > 0$ by Lemma \ref{Rowilastrow}.  Now $\lambda^*, \mu^* \in \mcP_{n-\mu_l}$, and $\lambda^* \unrhd \mu^*$ since $\lambda \unrhd \mu$.  If $T^*$ is any semistandard Young tableau of shape $\lambda^*$ and weight $\mu^*$, then we can add $\mu_l$ boxes with entries of $l$ to row $h$ of $T^*$, and we obtain a semistandard Young tableau of shape $\lambda$ and weight $\mu$.  This must be $T$, since $K_{\lambda \mu} = 1$, and then we must also have $K_{\lambda^* \mu^*} = 1$.   Let $0 = i_0^* < i_1^*< \cdots < i_t^* = \ell(\mu^*) = l-1$ be the indices which exist by the induction hypothesis, with the subpartitions $\lambda^{*k}$ and $\mu^{*k}$ with the desired properties.   Define $i_k = i_k^*$ if $k < t$, and $i_t = i_t^*+1 = l$.  Then $\lambda^k=\lambda^{*k}$ and $\mu^k = \mu^{*k}$ if $k <t$, so the desired properties hold.  When $k=t$, then $\lambda^k \unrhd \mu^k$, and the fact that $\lambda^k$ is of the desired shape follows from the assumption that either $l > h+1$, in which case $i_{t-1} = h-1$ necessarily, or $\lambda_h = \lambda_{h-1}$ with $l = h+1$, in which case $i_t^* = h$ and $i_t = h+1=l$.

Finally, consider the case with $\mu_{h+1} = \mu_h$, $l = h+1$, and $\lambda_h < \lambda_{h-1}$.  Assume that in row $h-1$ of $T$, there are some $h$ entries.  This implies that the right-most entry in row $h-1$ is $h$, since all $\mu_l$ of the $l=h+1$ entries in $T$ are in row $h$ of $T$, by Lemma \ref{Rowilastrow}.  Consider the left-most $l$ entry in row $h$ of $T$.  There cannot be an $l-1=h$ entry directly above this entry, since then all entries to the right of it would also be $h$, which would imply $\mu_h > \mu_{h+1}$, a contradiction.  Then we can exchange the left-most $l = h+1$ entry in row $h$ with the right-most $h$ entry in row $h-1$ of $T$ to obtain a distinct semistandard Young tableau $T'$ of shape $\lambda$ and weight $\mu$, contradicting $K_{\lambda \mu} = 1$.  Thus, there are no $h$ entries in row $h-1$ of $T$, and thus no $h$ entries above row $h$ of $T$ by the greedy filling of horizontal strips, since $\lambda_h < \lambda_{h-1}$, and if there are other $h$ entries, there must be some in row $h-1$.  If we define $\lambda^*$ by $\lambda^*_i = \lambda_i$ if $i < h$ and $\lambda^*_h = 0$, and we define $\mu^*$ by $\mu^*_i = \mu_i$ if $i < h$ and $\mu^*_h = \mu^*_{h+1} = 0$, then we can apply the induction hypothesis to $\lambda^*$ and $\mu^*$.  Taking the indices for $\lambda^*$ and $\mu^*$ to be $i_0$ through $i_{t-1} = h-1$ for $\lambda$ and $\mu$, and $i_t = h+1 = l$, gives the desired indices for $\lambda$ and $\mu$.  This completes the induction.
\end{proof}

 Given $\lambda \in \mcP_n$ and a composition $\omega$ of $n$, Narayanan \cite{Na06} considered the computational complexity of answering the question of whether $K_{\lambda \omega} > 0$, and proved that this can be answered in polynomial time in terms of the number of bits used to describe the tuples $\lambda$ and $\omega$ \cite[Proposition 1]{Na06}.  We now observe a similar result for the question of whether $K_{\lambda \omega} = 1$.  Following \cite{Na06}, we let $\mathrm{size}(\lambda, \omega)$ denote the number of bits used to describe $\lambda$ and $\omega$.

\begin{corollary} \label{K1comp} Given any $\lambda \in \mcP_n$ and any composition $\omega$ of $n$, the question of whether $K_{\lambda \omega} = 1$ can be answered in polynomial time.
\end{corollary}
\begin{proof} First suppose that the composition $\omega$ is a partition $\mu \in \mcP_n$.  If we are given a set of indices which satisfy the conditions of Theorem \ref{BZ}, then to check whether $\lambda^k \unrhd \mu^k$ for each $k$ requires the same order of number of computations as it takes to check whether $\lambda \unrhd \mu$, which can be answered in time $O(\mathrm{size}(\lambda, \mu))$, as in \cite[Proof of Proposition 1]{Na06}.  We show that we do not have to search the set of all possibilities of indices which may satisfy the conditions of Theorem \ref{BZ}, but rather we may find a single set of indices while performing exactly the computations required to check whether $\lambda^k \unrhd \mu^k$ for only this set of indices.  The steps in the algorithm which accomplishes this are as follows.

The input is $\lambda = (\lambda_1, \lambda_2, \ldots, \lambda_l)$ and $\mu = (\mu_1, \mu_2, \ldots, \mu_l)$, where $\ell(\lambda) \leq \ell(\mu) = l$, so some $\lambda_j$ might be $0$ for some $j$.  We begin with $k=1$, $i_0=0$, $i=1$ and $\lambda^k = (\lambda_{i_{k-1} + 1}, \ldots, \lambda_i) = (\lambda_1)$ and $\mu_k = (\mu_{i_{k-1} + 1}, \ldots, \mu_i) = (\mu_1)$.  We first check that $\lambda^k$ is one of the desired shapes as in condition (2) of Theorem \ref{BZ}, and if it is not at any point, then the output is FALSE.  Otherwise, we check if the sum of the parts of $\lambda^k$ is greater than or equal to the sum of the parts of $\mu^k$, and if it is not at any point the output is FALSE.  Otherwise, we check if $|\lambda^k| = |\mu^k|$, and if $|\lambda^k| > |\mu^k|$, then if $i=l$ the output is FALSE, and otherwise we increase $i$ by $1$, and add a part to $\lambda^k$ and $\mu^k$, so $\lambda^k = (\lambda_1, \lambda_2)$ and $\mu^k = (\mu_1, \mu_2)$ and we repeat these steps.  If $|\lambda^k| = |\mu^k|$, then we have found an index, and let $i_k = i$, and we increase $k$ by $1$, and repeat the steps starting with $\lambda^k = (\lambda_{i_{k-1} + 1})$.  If we get through all of these steps through $i=l$, then the output is TRUE.  The main point is the following.  Suppose that for some $i_1$, for $\lambda^1 = (\lambda_1, \ldots, \lambda_{i_1})$ and $\mu^1 = (\mu_1, \ldots, \mu_{i_1})$, we have $\lambda^1 \unrhd \mu^1$ and $\lambda^1$ is one of the shapes in condition (2) of Theorem \ref{BZ}.   If $K_{\lambda \mu} = 1$, then we can always choose $i_1$ as our first index, and if $\lambda^*$ and $\mu^*$ consist of the parts of $\lambda$ and $\mu$ after part $i_1$, then also $K_{\lambda^* \mu^*} = 1$.  It is this fact that allows us to find indices through the algorithm above, giving that we can check whether $K_{\lambda \mu} = 1$ in time $O(\mathrm{size}(\lambda, \mu))$.

If we consider now an arbitrary composition $\omega$, then we first must find a permutation which permutes the parts of $\omega$ to a partition $\mu$, and again as in \cite[proof of Proposition 1]{Na06}, with this the question of whether or not $K_{\lambda \omega} = 1$ can be answered in time $O(\mathrm{size}(\lambda, \omega) \ln(\mathrm{size}(\lambda, \omega)))$.
\end{proof}

It follows from Theorem \ref{BZ}, or from direct observation, that for any partition $\lambda$ we have $K_{\lambda \lambda} = 1$.  The following result gives exactly which partitions $\lambda$ satisfy $K_{\lambda \mu} = 1$ for only $\mu = \lambda$.  The proof of this statement as it appears in \cite[Corollary 5.1]{GaGoVi12} follows from an enumeration of those $\mu$ which satisfy $K_{\lambda \mu} = 1$, although here we can obtain it as a corollary of Theorem \ref{BZ} and Lemma \ref{ColLemma}.  Instead of giving that proof here, we instead give a generalization of the result in Corollary \ref{J1Multi} and a proof independent of the following result. 

\begin{corollary}[Gates, Goldman, and Vinroot]  \label{J1partitions} Let $\lambda \in \mcP$.  The only $\mu \in \mcP$ which satisfies $K_{\lambda \mu} = 1$ is $\mu = \lambda$ if and only if $\lambda_i - \lambda_{i+1} \leq 1$ for all $i$ (where $\lambda_i = 0$ if $i > \ell(\lambda)$).  
\end{corollary}

Our main goal is to give conditions on an $r$-multipartition $[\lambda(j)]$ and a partition $\mu$ which are equivalent to $K_{[\lambda(j)] \mu} = 1$, and to give the generalizations of Corollaries \ref{K1comp} and \ref{J1partitions} to multipartitions.  Before we answer these questions, we make a departure into representation theory to find motivation to study these Kostka numbers.

\section{Permutation Characters of Wreath Products} \label{wreath}
Let $S_n$ denote the symmetric group on $n$ elements.  For any group $G$, and any subgroup of permutations $P \leq S_n$, we let $G \wr P$ denote the semidirect product $G^n \rtimes P$, where $G^n$ is the direct product of $n$ copies of $G$, and $P$ acts by inverse permutation of indices on elements in $G^n$.  That is, if $g = (g_1, \ldots, g_n) \in G^n$ and $\tau \in P$, then the action of $\tau$ on $g$ is given by
$${^\tau g} = (g_{\tau^{-1}(1)}, \ldots, g_{\tau^{-1}(n)}).$$
If $P = S_n$, then $G \wr S_n$ is the standard wreath product.

Let $A$ and $B$ be finite groups, with complex representations $(\pi, V)$ and $(\rho, W)$, with characters $\chi$ and $\eta$, respectively.  We let $\pi \odot \rho$ denote the external tensor product of representations, which is a representation of $A \times B$ acting on the space $V \otimes W$, with character given by $(\chi \odot \eta)(a,b) = \chi(a)\eta(b)$.  In particular, $\bigodot_{i=1}^n \pi = \pi^{\odot n}$ is a representation of $A^n$ acting on $\bigotimes_{i=1}^n V = V^{\otimes n}$, with character $\bigodot_{i=1}^n \chi = \chi^{\odot n}$.  If $(\pi, V)$ and $(\rho, W)$ are both representations of $A$, then we let $\pi \otimes \rho$ denote the internal tensor product of representations, which is a representation of $A$ acting on $V \otimes W$, with character given by $(\chi \otimes \eta)(a) = \chi(a)\eta(a)$.    

We let $\bone$ denote the trivial character of a finite group.  For any finite group $A$, let $\Irr(A)$ denote the collection of irreducible complex characters of $A$.    Given any character $\chi$, we let $\chi(1)$ denote its degree.  If $\chi$ and $\eta$ are complex-valued class functions of $A$, let $\langle \chi, \eta \rangle = 1/|A| \sum_{a \in A} \chi(a) \overline{\eta(a)}$ be the standard inner product.  For a subgroup $D \leq A$, and a character $\xi$ of $D$, we let $\Ind_D^A(\xi)$ denote the induced character from $D$ to $A$.  

As in Section \ref{Kostka}, for any $\lambda \in \mcP_n$, we let $\pi^{\lambda}$ be the irreducible representation of $S_n$ associated with $\lambda$, with character $\chi^{\lambda}$, in such a way that $\chi^{(n)} = \bone$ and $\chi^{(1, 1, \ldots, 1)}$ is the sign character.

\subsection{Irreducible characters of wreath products}

Let $G$ be any finite group, and let $G_n = G \wr S_n$.  The irreducible complex characters of $G_n$ were first described by Specht \cite{Sp32}, and the description we give in this section and the next follows \cite[Section 4.3]{JaKe81}, \cite[Chapter I, Appendix B]{Ma95}, and we have also found \cite[Section 4]{Mar12} helpful.

If $(\varrho, V)$ is a representation of $G$, then $\varrho^{\odot n}$ is a representation of $G^n$ acting on $V^{\otimes n}$.  We extend $\varrho^{\odot n}$ to a representation $\underline{\varrho^{\odot n}}$ of $G_n = G \wr S_n$ by defining, for $g = (g_1, \ldots, g_n) \in G^n$, $\tau \in S_n$, and $v_i \in V$, 
$$ \left( \underline{\varrho^{\odot n}} \right)(g, \tau)(v_1 \otimes \cdots \otimes v_n) = \left( \varrho(g_{1})v_{\tau^{-1}(1)} \otimes \cdots \otimes \varrho(g_{n}) v_{\tau^{-1}(n)} \right).$$
Given the representation $\pi^{\lambda}$, $\lambda \in \mcP_n$, of $S_n$, extend $\pi^{\lambda}$ to a representation $\underline{\pi^{\lambda}}$ of $G_n$ trivially, that is,
$$ \underline{\pi^{\lambda}}(g, \tau) = \pi^{\lambda}(\tau), \quad \text{ for } g \in G^n, \tau \in S_n.$$
Then, given a representation $\varrho$ of $G$ and $\lambda \in \mcP_n$, define the representation $\varrho \wr \lambda$ of $G_n$ by
$$ \varrho \wr \lambda = \left( \underline{\varrho^{\odot n}} \right) \otimes \underline{\pi^{\lambda}},$$
and if $\psi$ is the character of $\varrho$, we let $\psi \wr \lambda$ denote the character of $\varrho \wr \lambda$.

Suppose that $G$ has $r = |\Irr(G)|$ irreducible characters, and consider the set $\mcP_n[r]$ of $r$-multipartitions of $n$.  We would like to associate a partition to each $\psi \in \Irr(G)$, and while we could label each irreducible character of $G$ as $\psi_j$ for $1 \leq j \leq r$ and associate $\lambda(j)$ with $\psi_j$, we instead write $\lambda(\psi)$ for this partition in the general case for the sake of clarity.  That is, for the parameterization which we will describe, we write $[\lambda(\psi)] \in \mcP_n[r]$ with the understanding that each $\psi \in \Irr(G)$ is labelled by $1 \leq j \leq |\Irr(G)|$ in some order, and that $\lambda(\psi)$ is taken as $\lambda(j)$, where $\psi$ is labelled by $j$.

Given any $[\lambda(\psi)] \in \mcP_n[|\Irr(G)|]$, define the subgroup $G[\lambda(\psi)]$ of $G_n$ by
$$ G[\lambda(\psi)] = \prod_{\psi \in \Irr(G)} G_{|\lambda(\psi)|} = \prod_{\psi \in \Irr(G)} (G \wr S_{|\lambda(\psi)|}).$$
Then we have that 
$$ \bigodot_{\psi \in \Irr(G)} \psi \wr \lambda(\psi) \quad \text{is a character of} \;\; G[\lambda(\psi)].$$

We may now give the parameterization of irreducible characters of $G_n$.

\begin{theorem}[Specht \cite{Sp32}] \label{specht} The irreducible complex characters of $G_n$ may be parameterized by $\mcP_n[|\Irr(G)|]$, where for each $[\lambda(\psi)] \in \mcP_n[|\Irr(G)|]$, the irreducible character $\eta^{[\lambda(\psi)]}$ is given by
$$ \eta^{[\lambda(\psi)]} = \Ind_{G[\lambda(\psi)]}^{G_n} \left(\bigodot_{\psi \in \Irr(G)} \psi \wr \lambda(\psi) \right).$$
The degree of $\eta^{[\lambda(\psi)]}$ is given by
$$ \eta^{[\lambda(\psi)]}(1) = n! \prod_{\psi \in \Irr(H)} \frac{\left(\psi(1)^{|\lambda(\psi)|}\right) \left(\chi^{\lambda(\psi)}(1)\right)}{|\lambda(\psi)|!}.$$
\end{theorem}

\subsection{Symmetric functions}

Macdonald \cite[Chapter 1, Appendix B]{Ma95} gives a method of describing the irreducible characters of $G_n$ using symmetric function theory, which is an extension of the description of the character theory of the symmetric groups using symmetric functions \cite[Chapter 1.8]{Ma95}.  Another detailed treatment of symmetric functions associated with wreath products is given by Ingram, Jing, and Stitzinger \cite{InJiSt09}.

For each $\psi \in \Irr(G)$, let $Y_{\psi} = \{y_{i\psi} \, \mid \, i \geq 1 \}$ be an infinite set of indeterminates, where any two $y_{i \psi}$ commute.  For any partition $\lambda$, let $s_{\lambda}(Y_{\psi})$ be the Schur symmetric function in these variables.  For any sequence $[\lambda(\psi)] \in \mcP_n[|\Irr(G)|]$, define 
$$ s_{[\lambda(\psi)]} = \prod_{\psi \in \Irr(G)} s_{\lambda(\psi)}(Y_{\psi}).$$
Let $\Lambda_n = \C\spn\{ s_{[\lambda(\psi)]} \, \mid \, [\lambda(\psi)] \in \mcP_n[|\Irr(G)|] \}$ and $\Lambda = \oplus_n \Lambda_n$.  Then $\Lambda$ is a graded $\C$-algebra, where multiplication is standard multiplication of symmetric functions coming from polynomial multiplication.  Let $R(G_n)$ be defined as the space of $\C$-valued class functions of $G_n$, and let $R = \oplus_n R(G_n)$.  For any $m, n \geq 1$, we may embed $G_m \times G_n$ in $G_{m+n}$ by embedding $S_m \times S_n$ in $S_{m+n}$ as we did in Section \ref{Kostka}, and we identify $G_m \times G_n$ as a subgroup of $G_{m+n}$ in this way.  For $\alpha \in R(G_n)$ and $\beta \in R(G_m)$, define the product $\alpha \beta \in R(G_{m+n})$ by the induced class function 
$$\alpha \beta = \Ind_{G_m \times G_n}^{G_{m+n}}(\alpha \times \beta).$$
This multiplication makes $R$ a graded $\C$-algebra.  We only need parts of the main result in \cite[Chapter 1, Appendix B]{Ma95}, and the following is a portion of \cite[I.B, (9.7)]{Ma95}.

\begin{theorem} \label{WreathChar} Let $\eta^{[\lambda(\psi)]}$ be the irreducible character of $G_n = G \wr S_n$ corresponding to $[\lambda(\psi)] \in \mcP_n[|\Irr(G)|]$.  There is an isomorphism of graded $\C$-algebras
$$ \ch : R \rightarrow \Lambda,$$
which we may define by $\ch(\eta^{[\lambda(\psi)]}) = s_{[\lambda(\psi)]}$, and extend linearly.  In particular, for any $\alpha, \beta \in R$, $\ch(\alpha \beta) = \ch(\alpha)\ch(\beta)$.
\end{theorem}

For any $|\Irr(G)|$-multipartition $[\lambda(\psi)]$, define the \emph{height} of the multipartition as
$$ \hgt([\lambda(\psi)]) = \mathrm{max} \{ \ell(\lambda(\psi)) \, \mid \, \psi \in \Irr(G) \}.$$
In particular, $\hgt([\gamma(\psi)]) = 1$ if and only if each $\gamma(\psi)$ is either the empty partition or has a single part (and at least one $\gamma(\psi)$ is not the empty partition).

For any $m \geq 1$ and any subset $\mcS \subseteq \Irr(G)$, define the element $H_m^{(\mcS)} \in \Lambda_m$ as 
$$ H_m^{(\mcS)}  = \sum_{|[\gamma(\psi)]| = m \atop {\hgt([\gamma(\psi)]) = 1 \atop{\gamma(\psi) = \varnothing \text{ if } \psi \not\in \mcS}}} s_{[\gamma(\psi)]}.$$
For any partition $\mu = (\mu_1, \mu_2, \ldots, \mu_l)$ of $n$, define $H_{\mu}^{(\mcS)} \in \Lambda_n$ as
$$ H_{\mu}^{(\mcS)} = H_{\mu_1}^{(\mcS)} H_{\mu_2}^{(\mcS)} \cdots H_{\mu_l}^{(\mcS)}.$$
We now prove a generalization of \eqref{KostkaSymm} in the algebra $\Lambda$.  This result seems to be equivalent to \cite[Corollary 4.6]{InJiSt09}, although we give a proof here for the purposes of self-containment.

\begin{lemma} \label{Hdecomp}
For any partition $\mu$ of $n$, we have
$$ H_{\mu}^{(\mcS)} = \sum_{|[\lambda(\psi)]| = n \atop{ \lambda(\psi) = \varnothing \text{ if } \psi \not\in \mcS}} K_{[\lambda(\psi)] \mu} s_{[\lambda(\psi)]}.$$
\end{lemma}
\begin{proof} Each factor $H^{(\mcS)}_{\mu_i}$ consists of summands $s_{[\gamma(\psi)]}$, such that each partition $\gamma(\psi)$ is either empty or a single row, so write $\gamma(\psi) = (\gamma(\psi)_1)$, where we take $\gamma(\psi)_1 = 0$ when $\gamma(\psi) = \varnothing$.  Then we have
$$ s_{[\gamma(\psi)]} = \prod_{\psi \in \Irr(G)} s_{(\gamma(\psi)_1)} (Y_{\psi}) = \prod_{\psi \in \Irr(G)} h_{\gamma(\psi)_1}(Y_{\psi}),$$
where $h_k$ is the complete symmetric function.  Now we have
\begin{align*}
H_{\mu}^{(\mcS)}  &= \prod_{i = 1}^l H_{\mu_i}^{(\mcS)} = \prod_{i=1}^l \left( \sum_{|[\gamma^{(i)}(\psi)]| = \mu_i \atop {\hgt([\gamma^{(i)}(\psi)]) = 1 \atop{\gamma^{(i)}(\psi) = \varnothing \text{ if } \psi \not\in \mcS}}} \prod_{\psi \in \Irr(G)} h_{\gamma^{(i)}(\psi)_1} (Y_{\psi}) \right) \\
& = \sum_{ 1 \leq i \leq l \atop{|[\gamma^{(i)}(\psi)]| = \mu_i \atop {\hgt([\gamma^{(i)}(\psi)]) = 1 \atop{\gamma^{(i)}(\psi) = \varnothing \text{ if } \psi \not\in \mcS}}}} \prod_{\psi \in \Irr(G)} h_{\gamma^{(1)}(\psi)_1} (Y_{\psi}) \cdots h_{\gamma^{(l)}(\psi)_1} (Y_{\psi}) \\
& = \sum_{ 1 \leq i \leq l \atop{|[\gamma^{(i)}(\psi)]| = \mu_i \atop {\hgt([\gamma^{(i)}(\psi)]) = 1 \atop{\gamma^{(i)}(\psi) = \varnothing \text{ if } \psi \not\in \mcS}}}} \prod_{\psi \in \Irr(G)} \sum_{ \lambda(\psi) \in \mcP_{|\omega(\psi)|}} K_{\lambda(\psi) \omega(\psi)} s_{\lambda(\psi)} (Y_{\psi}),
\end{align*}
where we take $\omega(\psi)$ to be the composition $\omega(\psi) = (\gamma^{(1)}(\psi)_1, \ldots, \gamma^{(l)}(\psi)_1)$, and we have applied \eqref{KostkaSymm}.  Note that we have $\sum_{\psi} \omega(\psi) = \mu \vdash n$.  We may now rearrange sums and products to obtain
\begin{align*}
H_{\mu}^{(\mcS)} & = \sum_{|[\lambda(\psi)]| = n \atop{ \lambda(\psi) = \varnothing \text{ if } \psi \not\in \mcS}}  \left( \sum_{ \sum_{\psi} \omega(\psi)  = \mu \atop{ |\omega(\psi)| = |\lambda(\psi)|}} \prod_{\psi \in \Irr(G)} K_{\lambda(\psi) \omega(\psi)} \right) s_{[\lambda(\psi)]} \\
& = \sum_{|[\lambda(\psi)]| = n \atop{ \lambda(\psi) = \varnothing \text{ if } \psi \not\in \mcS}} K_{[\lambda(\psi)] \mu} s_{[\lambda(\psi)]},
\end{align*}
where we have applied Proposition \ref{compdecomp} to obtain the last equality.
 \end{proof}

\subsection{Permutation characters of $G_n$ for $G$ abelian}

We now take $G$ to be a finite abelian group, and $G_n = G \wr S_n$.  So, each $\psi \in \Irr(G)$ satisfies $\psi(1) = 1$, and so for any $\eta^{[\lambda(\psi)]} \in \Irr(G_n)$ it follows from Theorem \ref{specht} that
$$ \eta^{[\lambda(\psi)]}(1) = n! \prod_{\psi \in \Irr(G)} \frac{\chi^{\lambda(\psi)}(1)}{|\lambda(\psi)|!}.$$
Let $L \leq G$ be any subgroup of $G$, and consider $L_n = L \wr S_n$ as a subgroup of $G_n$.  The decomposition of the permutation character of $G_n$ on $L_n$ is as follows.

\begin{proposition} \label{PermDecomp} Let $G$ be a finite abelian group and $L \leq G$ a subgroup, so $L_n \leq G_n$.  Then
$$\Ind_{L_n}^{G_n}(\bone) = \sum_{|[\gamma(\psi)]| = n \atop {\hgt([\gamma(\psi)]) = 1 \atop{\gamma(\psi) = \varnothing \text{ if } L \not\subseteq \mathrm{ker}(\psi)}}} \eta^{[\gamma(\psi)]}.$$
\end{proposition}
\begin{proof} We first show that the character degrees on both sides are the same.  The degree of $\Ind_{L_n}^{G_n}(\bone)$ is $[G_n : L_n] = |G/L|^n$.  Consider one character $\eta^{|\gamma(\psi)|}$ in the sum on the right.  Since each $\gamma(\psi)$ has a single part or is empty, then $\chi^{\gamma(\psi)}(1) = 1$.  Since $L \subseteq \mathrm{ker}(\psi)$ whenever $\gamma(\psi) \neq \varnothing$, then we may think of $\psi$ as a character of $G/L$, and the multipartitions we are considering may be thought of as $|\Irr(G/L)|$-multipartitions.  We now have
$$ \eta^{|\gamma(\psi)|}(1) = \frac{n!}{\prod_{\psi \in \Irr(G/L)} |\gamma(\psi)|!},$$
which is a multinomial coefficient.  In summing the degrees of these characters, we are summing over all possible multinomial coefficients of degree $n$ with $|G/L|$ terms, which has sum $|G/L|^n$.

Since the degrees are equal, it is now enough to show that any character $\eta^{[\gamma(\psi)]}$ in the sum appears with nonzero multiplicity in the permutation character.  Indeed, we will show that
\begin{equation} \label{want}
\left \langle \Ind_{L_n}^{G_n}(\bone), \eta^{[\gamma(\psi)]} \right \rangle = 1.
\end{equation}
From Theorem \ref{specht}, we have that
$$ \eta^{[\gamma(\psi)]} = \Ind_{G[\gamma(\psi)]}^{G_n} (\xi^{[\gamma(\psi)]}) \quad \text{where} \quad \xi^{[\gamma(\psi)]} = \bigodot_{\psi \in \Irr(G)} \psi \wr \gamma(\psi).$$
By Frobenius reciprocity, we have
$$ \left \langle \Ind_{L_n}^{G_n}(\bone), \eta^{[\gamma(\psi)]} \right \rangle = \left \langle \Res_{G[\gamma(\psi)]} ( \Ind_{L_n}^{G_n}(\bone)), \xi^{[\gamma(\psi)]} \right \rangle,$$
where $\mathrm{Res}$ is restriction of characters.  By Mackey's theorem, we have
$$\Res_{G[\gamma(\psi)]} ( \Ind_{L_n}^{G_n}(\bone)) = \sum_{g \in \left[G[\gamma(\psi)] \backslash G_n / L_n\right]} \Ind_{G[\gamma(\psi)] \cap gL_n g^{-1}}^{G[\gamma(\psi)]}(\bone),$$
where the sum is over a set of double coset representatives.  A short calculation reveals that the only double coset representative is $g = 1 \in G_n$.  Define 
$$L[\gamma(\psi)] = G[\gamma(\psi)] \cap L_n = \prod_{\psi \in \Irr(G)} (L \wr S_{|\gamma(\psi)|}).$$  
By Mackey's theorem and Frobenius reciprocity, we then have
\begin{align*}
\left \langle \Ind_{L_n}^{G_n}(\bone), \eta^{[\gamma(\psi)]} \right \rangle & =\left \langle \Ind_{L[\gamma(\psi)]}^{G[\gamma(\psi)]}(\bone), \xi^{[\gamma(\psi)]} \right \rangle \\
 & = \left \langle \bone, \Res_{L[\gamma(\psi)]} (\xi^{[\gamma(\psi)]}) \right \rangle \\
 & = \frac{1}{|L[\gamma(\psi)]||} \sum_{ x \in L[\gamma(\psi)]} \xi^{[\gamma(\psi)]}(x).
\end{align*}
We have $\gamma(\psi) = \varnothing$ unless $L \subseteq \mathrm{ker}(\psi)$, and we have $\chi^{\gamma(\psi)}$ is trivial for every $\psi \in \Irr(G)$.  It follows directly from the definition that for every $\psi \in \Irr(G)$ and every $(a, \tau) \in L_{|\gamma(\psi)|} = L \wr S_{|\gamma(\psi)|}$ that
$$ (\psi \wr \gamma(\psi))(a, \tau) = 1.$$
Thus, for any $x \in L[\gamma(\psi)]$, we have $\xi^{[\gamma(\psi)]}(x) = 1$, and \eqref{want} follows.
\end{proof}

We now see multipartition Kostka numbers as multiplicities, in the following generalization of \eqref{KostkaPerm} to wreath products.

\begin{theorem} \label{YoungWreath} Let $G$ be an abelian group, $L \leq G$ a subgroup, $n \geq 1$, and $\mu = (\mu_1, \ldots, \mu_l) \in \mcP_n$.  Define $L_{\mu} = L \wr S_{\mu} = \prod_{i=1}^l (L \wr S_{\mu_i})$.  Then the permutation character of $G_n$ on $L_{\mu}$ decomposes as
$$ \Ind_{L_{\mu}}^{G_n}(\bone) = \sum_{|[\lambda(\psi)]| = n \atop{ \lambda(\psi) = \varnothing \text{ if } L \not\subseteq \mathrm{ker}(\psi)}} K_{[\lambda(\psi)]\mu}\eta^{[\lambda(\psi)]}.$$
\end{theorem}
\begin{proof} Let $\mcS = \{ \psi \in \Irr(G) \, \mid \, L \subseteq \mathrm{ker}(\psi) \}$.  For any $m \geq 1$, we have from Proposition \ref{PermDecomp} and Theorem \ref{WreathChar} that
$$\ch(\Ind_{L_m}^{G_m}(\bone)) = H_m^{(\mcS)}.$$
Then we have
\begin{align*}
\Ind_{L_{\mu}}^{G_n}(\bone) & = \Ind_{L_{\mu_1} \times \cdots \times L_{\mu_l}}^{G_n}(\bone) \\
& = \Ind_{G_{\mu_1} \times \cdots \times G_{\mu_l}}^{G_n} \left(\Ind_{L_{\mu_1}}^{G_{\mu_1}}(\bone) \cdots \Ind_{L_{\mu_l}}^{G_{\mu_l}}(\bone) \right).
\end{align*}
From these two statements, and Theorem \ref{WreathChar}, we have $\ch(\Ind_{L_{\mu}}^{G_n}(\bone)) = H_{\mu}^{(\mcS)}$.  The result now follows by applying Lemma \ref{Hdecomp} and Theorem \ref{WreathChar}.
\end{proof}

An example of particular interest is the case that $G$ is a cyclic group of order $r$, which we denote by $C^{(r)}$.  In this case $G_n = C^{(r)} \wr S_n$ is the complex reflection group $G(r, 1, n)$.  We realize $C^{(r)}$ as the multiplicative group of $r$th roots of unity, with generator $\zeta = e^{2\sqrt{-1}\pi/r}$, so that $C^{(r)} = \{ \zeta^{j-1} \, \mid \, 1 \leq j \leq r \}$.  We label the irreducible characters of $C^{(r)}$ by $\psi_j(c) = c^{j-1}$, for $1 \leq j \leq r$.  Then the irreducible characters of $G_n$ are parameterized by $\mcP_n[r]$, and we write $\lambda(j)$ for the partition corresponding to $\psi_j \in \Irr(C^{(r)})$.  The subgroups of $C^{(r)}$ correspond to positive divisors $d|r$, where $C^{(d)} = \langle \zeta^{r/d} \rangle$ is the subgroup of $C^{(r)}$ of order $d$.  Then we have $C^{(d)} \subseteq \mathrm{ker}(\psi_j)$ if and only if $d|(j-1)$.  Proposition \ref{PermDecomp} and Theorem \ref{YoungWreath} then give the following.

\begin{corollary} \label{CRGinduce} Let $r, n \geq 1$, and let $d$ be a positive divisor of $r$.  Then
$$ \Ind_{C^{(d)} \wr S_n}^{C^{(r)} \wr S_n}(\bone) = \sum_{[\gamma(j)] \in \mcP_n[r] \atop{ \hgt([\gamma(j)]) = 1 \atop{ \gamma(j) = \varnothing \text{ if } d \nmid (j-1)}}} \eta^{[\gamma(j)]}.$$
More generally, if $\mu \in \mcP_n$, then
$$ \Ind_{C^{(d)} \wr S_{\mu}}^{C^{(r)} \wr S_n} (\bone) = \sum_{[\lambda(j)] \in \mcP_n[r] \atop {\lambda(j) = \varnothing \text{ if } d \nmid (j-1)}} K_{[\lambda(j)] \mu} \eta^{[\lambda(j)]}, \quad  \text{and} \quad \Ind_{S_{\mu}}^{C^{(r)} \wr S_n}(\bone) = \sum_{[\lambda(j)] \in \mcP_n[r]} K_{[\lambda(j)] \mu} \eta^{[\lambda(j)]},$$
where the second case is the result of taking $d=1$ in the first case.
\end{corollary}

A few cases of the induced representations in Corollary \ref{CRGinduce} appear as parts of the generalized involution models for the complex reflection groups $G(r, 1, n)$ constructed by Marberg \cite{Mar12}.  Namely, $\Ind_{S_n}^{C^{(r)} \wr S_n}(\bone)$ is the piece of the model corresponding to $k=0$ for $r$ odd, and $\Ind_{C^{(2)} \wr S_n}^{C^{(r)} \wr S_n}(\bone)$ corresponds to the $k= \ell = 0$ piece for $r$ even in \cite[Theorem 5.6]{Mar12}.

\section{Multiplicity One for Multipartitions} \label{Multi1}

In this section we prove the main result, which is a generalization of Theorem \ref{BZ} to multipartitions.  We begin with a necessary condition for a multipartition Kostka number to be $1$.

\begin{lemma} \label{KosSeq1tilde} Suppose $[\lambda(j)] \in \mcP_n[r]$ and $\mu \in \mcP_n$, and let $\tlam \in \mcP_n$ be as in \eqref{tildedef}.  Then $K_{[\lambda(j)] \mu} \geq K_{\tlam \mu}$.  In particular, if $K_{[\lambda(j)] \mu} = 1$ then $K_{\tlam \mu} = 1$.
\end{lemma}
\begin{proof} From Lemma \ref{KosSeqNonzero}, if $K_{\tlam \mu} = 0$, then $K_{[\lambda(j)] \mu} = 0$, so we may suppose that $K_{\tlam \mu} > 0$.  Let $T$ be some semistandard Young tableau of shape $\tlam$ and weight $\mu$, and we have seen in Section \ref{MultiDefs} after Lemma \ref{KosSeqNonzero}, that from $T$ we can construct a semistandard Young $r$-multitableau $[T(j)]$ of shape $[\lambda(j)]$ and weight $\mu$.  Now suppose $T'$ is a semistandard Young tableaux of shape $\tlam$ and weight $\mu$ which is distinct from $T$.  Then some column, say the $i$th column, of $T$ is distinct from column $i$ of $T'$.  Suppose that the entries of column $i$ of $T$ are the entries of column $k$ of $T(j)$ of an $r$-multitableau $[T(j)]$ of shape $[\lambda(j)]$.  Through the process described in Section \ref{MultiDefs}, we can construct a semistandard Young $r$-multitableau $[T'(j)]$ with the entries of column $i$ of $T'$ as the entries of column $k$ of $T'(j)$ (since the order of entries is still preserved, since it was preserved in constructing $[T(j)]$).  In particular, $[T(j)]$ and $[T'(j)]$ are distinct.  It follows that $K_{[\lambda(j)] \mu} \geq K_{\tlam \mu}$.  If $K_{[\lambda(j) \mu]} = 1$, then $K_{\tlam \mu} > 0$ by Lemma \ref{KosSeqNonzero}, and so $K_{\tlam \mu} = 1$.
\end{proof}

The converse of the second statement of Lemma \ref{KosSeq1tilde} is false, which we can see by considering $[\lambda(j)]$ defined by $\lambda(1) = \lambda(2) =(1)$ and $\mu=(1,1)$.  Then $\tlam = (2)$, and $K_{\tlam \mu} = 1$ where $T = \young(12)$ is the unique semistandard Young tableau of shape $\tlam$ and weight $\mu$.  However, $[T(j)]$ and $[T'(j)]$ defined by  $T(1) = \young(1)$, $T(2) = \young(2)$, and $T'(1) = \young(2)$, $T'(2) = \young(1)$, are both semistandard Young $2$-multitableaux of shape $[\lambda(j)]$ and weight $\mu$.

The following observation is crucial in the proof of the main result.

\begin{lemma} \label{SameCol} Let $[\lambda(j)] \in \mcP_n[r]$ and $\mu \in \mcP_n$ such that $K_{[\lambda(j)] \mu} = 1$.  Let $[T(j)]$ be the unique semistandard Young $r$-multitableau of shape $[\lambda(j)]$ and weight $\mu$.  If two distinct $\lambda(j)$ have some column of the same length, then every column of that length in $[T(j)]$ must have identical entries.
\end{lemma}
\begin{proof} Suppose that two distinct $\lambda(j)$ have some column of the same length, but two columns of that length do not have identical entries in $[T(j)]$.  By Lemma \ref{KosSeq1tilde}, $K_{\tlam \mu} = 1$, and let $T$ be the unique semistandard Young tableau of shape $\tlam$ and weight $\mu$.  By the process described in Section \ref{MultiDefs}, we can construct $[T(j)]$ from $T$ by making each column of $T$ the same as some column of some $T(j)$.  So, some two columns of the same length in $T$ do not have the same entries, say columns $i$ and $i+1$.  Suppose $\lambda(j_1)$ and $\lambda(j_2)$ are two distinct partitions of $[\lambda(j)]$ with columns of length the same as columns $i$ and $i+1$ of $\tlam$.  Again from the process from Section \ref{MultiDefs}, from the tableau $T$ of shape $\tlam$ and weight $\mu$, we can construct a semistandard Young $r$-multitableau of shape $[\lambda(j)]$ and weight $\mu$ by taking columns of $T$ and making them columns of any individual tableau of shape some $\lambda(j)$, as long as left-to-right order is preserved in individual tableau.  This means we can take column $i$ of $T$ and make it a column in a tableau of shape $\lambda(j_1)$ and column $i+1$ of $T$ a column in a tableau of shape $\lambda(j_2)$, or vice versa.  This implies $K_{[\lambda(j)] \mu} > 1$, a contradiction.
\end{proof}

We may now prove the main theorem.  We note that in the bipartition ($r=2$) case, this is presumably implied by the case of a type $B/C$ Lie algebra in the result of Berenshte\u{\i}n and Zelevinski\u{\i} \cite[Theorem 1.2]{BeZe90}.

\begin{theorem} \label{MultiMultOne} Let $[\lambda(j)] \in \mcP_n[r]$ and $\mu \in \mcP_n$ with $\ell(\mu) = l$.  Then $K_{[\lambda(j)] \mu} = 1$ if and only if there exists a choice of indices $0 = i_0 < i_1 < \cdots < i_t = l$ such that, for each $k = 1, \ldots, t$, the $r$-multipartition $[\lambda(j)^k]$ defined by
$$ \lambda(j)^k = (\lambda(j)_{i_{k-1} + 1}, \lambda(j)_{i_{k-1}+2}, \ldots, \lambda(j)_{i_k}),$$
where $\lambda(j)_i = 0$ if $i > \ell(\lambda(j))$, and the partition $\mu^k = (\mu_{i_{k-1} + 1}, \mu_{i_{k-1} + 2}, \ldots, \mu_{i_k})$, satisfy the following:
\begin{enumerate} 
\item[(1)] $\widetilde{\lambda^k} \unrhd \mu^k$, and
\item[(2)] for at most one $j$, either $\lambda(j)_{i_{k-1} + 1} = \lambda(j)_{i_{k-1} + 2}= \cdots= \lambda(j)_{i_k - 1} > \lambda(j)_{i_k}$ or $\lambda(j)_{i_{k-1}+1} > \lambda(j)_{i_{k-1} + 2}= \lambda(j)_{i_{k-1} + 3}= \cdots= \lambda(j)_{i_k}$, and for all other $j$, $\lambda(j)_{i_{k-1} + 1} = \lambda(j)_{i_{k-1} + 2}= \cdots= \lambda(j)_{i_k - 1} = \lambda(j)_{i_k}$.
\end{enumerate}
\end{theorem}
\begin{proof}  
We first assume that $[\lambda(j)]$ and $\mu$ satisfy the listed conditions, and show that $K_{[\lambda(j)]\mu} = 1$.  This direction of the proof parallels the same direction of the proof of Theorem \ref{BZ}.  Since $\widetilde{\lambda^k} \unrhd \mu^k$ for $k = 1, \ldots, t$, then $\tlam \unrhd \mu$, and so $K_{[\lambda(j)] \mu} > 0$ by Lemma \ref{KosSeqNonzero}.  Let $[T(j)]$ be some $r$-multitableau of shape $[\lambda(j)]$ and weight $\mu$.  The entries in each column of each $T(j)$ strictly increase, and so the entries in row $i$ of any $T(j)$ must be all at least $i$.  Since row $i_1 + 1$ of each $T(j)$ must have entries at least $i_1 + 1$, then all entries $1$ through $i_1$ must appear in rows $1$ through $i_1$ of $[T(j)]$.  Since $|[\lambda(j)^1]| = |\mu^1| = \mu_1 + \cdots + \mu_{i_1}$, then also the only entries in rows $1$ through $i_1$ of $[T(j)]$ can be $1$ through $i_1$.  By induction, for each $k$, the only entries in rows $i_{k-1} + 1$ through $i_k$ of $[T(j)]$ are $i_{k-1} + 1$ through $i_k$, and these are the only rows in which these entries appear.  Thus, we may consider each $[\lambda(j)^k]$ with $\mu^k$ independently, and $K_{[\lambda(j)] \mu} = \prod_{k=1}^t K_{[\lambda(j)^k] \mu^k}$.  

It is now enough to show that $K_{[\lambda(j)] \mu} = 1$ whenever $\tlam \unrhd \mu$, $\ell(\mu) = l$, and for at most one $j$ either $\lambda(j)_1 = \lambda(j)_2 = \cdots \lambda(j)_{l-1} > \lambda(j)_l$ or $\lambda(j)_1 > \lambda(j)_2 = \lambda(j)_3 = \cdots = \lambda(j)_l$, and for all other $j$, $\lambda(j)_1 = \lambda(j)_2 = \cdots = \lambda(j)_l$.  As above, let $[T(j)]$ denote some $r$-multitableau of shape $[\lambda(j)]$ and weight $\mu$.  For all $j$ such that $\lambda(j)_1 = \lambda(j)_2 = \cdots = \lambda(j)_l$, then $\ell(\lambda(j)) = l$ (or $\lambda(j) = (0)$), and every column of $T(j)$ must have the entries $1$ through $l$ in sequential order (when $\lambda(j) \neq (0)$).  If these are the only nonempty $\lambda(j)$, then $[T(j)]$ is uniquely determined and we are done.  Otherwise, consider the unique $j =j'$ such that $\lambda(j')_1 >\lambda(j')_2 = \lambda(j')_3 = \cdots = \lambda(j')_l$ or $\lambda(j')_1=\lambda(j')_2 = \cdots = \lambda(j')_{l-1} > \lambda(j')_l$.    Let $s$ be the total number of columns in all other nonempty $\lambda(j)$, and define the partition $\nu$ by $\nu_i = \mu_i - s$.  Then $\lambda(j') \unrhd \nu$.  Since all other $T(j)$ are uniquely determined, it is enough to show that $K_{\lambda(j') \nu} = 1$.  This is implied by Lemma \ref{ColLemma}.

We now assume that $K_{[\lambda(j)] \mu} = 1$.  Throughout, let $[T(j)]$ be the unique $r$-multitableau of shape $[\lambda(j)]$ and weight $\mu$.  By Lemma \ref{KosSeq1tilde}, we have $K_{\tlam \mu} = 1$.  By Theorem \ref{BZ}, there exists a choice of indices, say $0 < \iota_0 < \iota_1 < \cdots < \iota_u = l$ such that, for each $k = 1, \ldots, u$, we have
$$ \tlam^k = (\tlam_{\iota_{k-1} + 1}, \tlam_{\iota_{k-1}+2}, \ldots, \tlam_{\iota_k}) \quad \text{and} \quad \mu^k = (\mu_{\iota_{k-1}+1}, \mu_{\iota_{k-1}+2}, \ldots, \mu_{\iota_k}),$$
where $\tlam_i = 0$ if $i > \ell(\tlam)$, satisfy $\tlam^k \unrhd \mu^k$ and either 
$$\tlam_{\iota_{k-1}+1} = \tlam_{\iota_{k-1}+2} = \cdots = \tlam_{\iota_k -1} \quad  \text{or} \quad \tlam_{\iota_{k-1} +1} > \tlam_{\iota_{k-1}+2} = \tlam_{\iota_{k-1} + 3} = \cdots = \tlam_{\iota_k}.$$
For each $k = 1, \ldots, u$, consider the $r$-multipartition $[\lambda(j)^k]$ given by
$$\lambda(j)^k = (\lambda(j)_{\iota_{k-1}+1}, \lambda(j)_{\iota_{k-1}+2}, \ldots, \lambda(j)_{\iota_k}),$$
where $\lambda(j)_{\iota} = 0$ if $\iota > \ell(\lambda(j))$.  Then we have $\widetilde{\lambda^k} = \tlam^k$, so that $\widetilde{\lambda^k} \unrhd \mu^k$.

Note that whenever $\tlam_i = \tlam_{i+1}$ then $\lambda(j)_i = \lambda(j)_{i+1}$ for each $j$, since $\lambda(j)_i \geq \lambda(j)_{i+1}$ for each $j$, and if $\lambda(j)_i > \lambda(j)_{i+1}$ for some $j$, then $\tlam_i = \sum_j \lambda(j)_i > \sum_j \lambda(j)_{i+1} = \tlam_{i+1}$.  Thus, for each $k = 1, \ldots, u$,  we have
$$\lambda(j)_{\iota_{k-1}+1} = \lambda(j)_{\iota_{k-1}+2} = \cdots = \lambda(j)_{\iota_k -1} \quad \text{or} \quad \lambda(j)_{\iota_{k-1}+2} = \lambda(j)_{\iota_{k-1} + 3} = \cdots = \lambda(j)_{\iota_k},$$
where either the first holds for every $j$ or the second holds for every $j$.
If for each $k$, there is at most one $j$ such that either $\lambda(j)_{\iota_{k-1} + 1} = \lambda(j)_{\iota_{k-1} + 2}= \cdots= \lambda(j)_{\iota_k - 1} > \lambda(j)_{\iota_k}$ or $\lambda(j)_{\iota_{k-1}+1} > \lambda(j)_{\iota_{k-1} + 2}= \lambda(j)_{\iota_{k-1} + 3}= \cdots= \lambda(j)_{\iota_k}$, then we can take $t=u$, and $i_k = \iota_k$ for $k=1, \ldots, t$, and the desired conditions are satisfied.  Otherwise, we consider the following possibilities.

Consider any $k$ such that there are at least two $j$ such that $\lambda(j)_{\iota_{k-1}+1} > \lambda(j)_{\iota_{k-1}+2}$, or at least two $j$ such that $\lambda(j)_{\iota_k - 1} > \lambda(j)_{\iota_k}$.  Suppose first that $\lambda(j)_{\iota_{k-1}+1} > \lambda(j)_{\iota_{k-1}+2}$ for at least two $j$.  Then for each such $j$, there is at least one column of length $\iota_{k-1} + 1$ in $\lambda(j)$.  By Lemma \ref{SameCol}, every column of length $\iota_{k-1}+1$ must be filled identically in each $T(j)$.  This means that the right-most $\lambda(j)_{\iota_{k-1} + 1} - \lambda(j)_{\iota_{k-1} + 2}$ entries of row $\iota_{k-1}+1$ of $T(j)$ must be identical.  Consider those $j$ for which $\lambda(j)_{\iota_{k-1}+2} > 0$.  Then $\ell(\lambda(j)^k) = \ell(\mu^k)$ for these $j$, in which case the left-most $\lambda(j)_{\iota_{k-1}+2}$ columns of $\lambda(j)^k$, in $T(j)$, must contain the entries $\iota_{k-1} + 1$ through $\iota_k$ in sequential order, as there is no other choice.  Of these, for the $j$ satisfying $\lambda(j)_{\iota_{k-1}+1} > \lambda(j)_{\iota_{k-1}+2}$, the entries in the rest of row $\iota_{k-1}+1$ in $T(j)$ must also be $\iota_{k-1}+1$, since $\mu_{\iota_{k-1}+1} \geq \mu_{\iota_{k-1}+2}$.  For those $j$ where $\lambda(j)_{\iota_{k-1}+2} = 0$, since all entries in $T(j)$ of row $\iota_{k-1}+1$ must be the same, then every entry must also be $\iota_{k-1}+1$.  We have now shown that in this situation, every entry in row $\iota_{k-1}+1$ must be $\iota_{k-1}+1$ in every $T(j)$, and no other $\iota_{k-1}+1$ entries appear.  In this case, we define $[\lambda(j)^{k^-}]$, $[\lambda(j)^{k^+}]$, $\mu^{k-}$, and $\mu^{k+}$ by
$$ \lambda(j)^{k^-} = (\lambda(j)_{\iota_{k-1}+1}), \quad \lambda(j)^{k^+} = (\lambda(j)_{\iota_{k-1}+2}, \ldots, \lambda(j)_{\iota_k}),$$
$$ \mu^{k^-} = (\mu_{\iota_{k-1} + 1}), \quad \text{and} \quad \mu^{k^+} = (\mu_{\iota_{k-1} + 2}, \ldots, \mu_{\iota_k}),$$
and we define $\iota_{k^-} = \iota_{k-1}+1$ and $\iota_{k^+} = \iota_k$.  Note that for every $j$, the parts of $\lambda(j)^{k^+}$ satisfy $\lambda(j)_{\iota_{k-1} + 2} = \cdots = \lambda(j)_{\iota_k}$, and we also have $\widetilde{\lambda^{k^-}} \unrhd \mu^{k^-}$ and $\widetilde{\lambda^{k^+}} \unrhd \mu^{k^+}$.

Next suppose that $\lambda(j)_{\iota_k - 1} > \lambda(j)_{\iota_k}$ for at least two $j$.  The analysis of this case is very similar to the above.  We know that there is a column of length $\iota_k - 1$ in at least two $\lambda(j)$, and so by Lemma \ref{SameCol}, every column of this length must be filled identically in every $T(j)$.  For those $j$ such that $\lambda(j)_{\iota_k} > 0$, we have $\ell(\lambda(j)^k) = \ell(\mu^k)$, and so the left-most $\lambda(j)_{\iota_k}$ columns of $\lambda(j)^k$ must have the sequential entries $\iota_{k-1}+1$ through $\iota_k$ in $T(j)$.  It then follows that the right-most $\lambda(j)_{\iota_k -1} - \lambda(j)_{\iota_k}$ columns in $\lambda(j)^k$ must have the sequential entries $\iota_{k-1}+1$ through $\iota_k - 1$ in every $T(j)$, in order for $\mu^k$ to be a partition.  This must also then hold for those $j$ where $\lambda(j)_{\iota_k} = 0$.  Thus, every entry in row $\iota_k$ of every $T(j)$ must be $\iota_k$, and no other $\iota_k$ entries appear in any $T(j)$.  In this case, we define $[\lambda(j)^{k^-}]$, $[\lambda(j)^{k^+}]$, $\mu^{k^-}$, and $\mu^{k^+}$ by
$$ \lambda(j)^{k^-} = (\lambda(j)_{\iota_{k-1}+1}, \lambda(j)_{\iota_{k-1}+2}, \ldots, \lambda(j)_{\iota_k -1}), \quad \lambda(j)^{k^+} = (\lambda(j)_{\iota_k}),$$
$$ \mu^{k^-} = (\mu_{\iota_{k-1} + 1}, \mu_{\iota_{k-1} + 2}, \ldots, \mu_{\iota_k -1}), \quad \text{and} \quad \mu^{k^+} = (\mu_{\iota_k}),$$
and we define $\iota_{k^-} = \iota_{k} - 1$ and $\iota_{k^+} = \iota_k$.  For every $j$, the parts of $\lambda(j)^{k^-}$ satisfy $\lambda(j)_{\iota_{k-1}+1} = \cdots = \lambda(j)_{\iota_k -1}$, and we again have $\widetilde{\lambda^{k^-}} \unrhd \mu^{k^-}$ and $\widetilde{\lambda^{k^+}} \unrhd \mu^{k^+}$.

We define $[\lambda(j)^{k^-}]$, $[\lambda(j)^{k^+}]$, $\mu^{k^-}$, $\mu^{k^+}$, $\iota_{k^-}$, and $\iota_{k^+}$ as above for every $k$ with the property that there are either at least two $j$ such that $\lambda(j)_{\iota_{k-1}+1} > \lambda(j)_{\iota_{k-1} + 2}$ or at least two $j$ such that $\lambda(j)_{\iota_k - 1} > \lambda(j)_{\iota_k}$.  For all other $k$, we define $[\lambda(j)^{k^-}] = [\lambda(j)^{k^+}] = [\lambda(j)^k]$, $\mu^{k^-} = \mu^{k^+} = \mu^k$, and $\iota_{k^+} = \iota_{k^-} = \iota_k$.  We then define the indices
$$ \{ i_k \, \mid \, k = 1, \ldots, t \} = \{ \iota_{k^-}, \iota_{k^+} \, \mid \, k = 1, \ldots, u \},$$
where $t$ is equal to the sum of $u$ and number of $k$ such that there are at least two $j$ such that $\lambda(j)_{\iota_{k-1}+1} > \lambda(j)_{\iota_{k-1} + 2}$ or at least two $j$ such that $\lambda(j)_{\iota_k - 1} > \lambda(j)_{\iota_k}$.  By construction, we now have that the redefined $[\lambda(j)^k]$ and $\mu^k$ for $k=1, \ldots, t$, given by 
$$ \lambda(j)^k = (\lambda(j)_{i_{k-1} + 1}, \lambda(j)_{i_{k-1}+2}, \ldots, \lambda(j)_{i_k}),$$
where $\lambda(j)_i = 0$ if $i > \ell(\lambda(j))$, and
$$ \mu^k = (\mu_{i_{k-1} + 1}, \mu_{i_{k-1} + 2}, \ldots, \mu_{i_k}),$$
satisfy the required conditions. \end{proof}

In terms of computational complexity, we can give the following result generalizing the result of Narayanan \cite[Proposition 1]{Na06} and Corollary \ref{K1comp}.

\begin{corollary} \label{Kmulticomp} For any $[\lambda(j)] \in \mcP_n[r]$ and any composition $\omega$ of $n$, the questions of whether $K_{[\lambda(j)] \omega} > 0$ and $K_{[\lambda(j)] \omega} = 1$ can be answered in polynomial time.
\end{corollary}
\begin{proof}  By Corollary \ref{PermInvar} and Lemma \ref{KosSeqNonzero}, we have $K_{[\lambda(j)] \omega} > 0$ if and only if $K_{\tlam \omega} > 0$, which can be checked in polynomial time in terms of $\mathrm{size}(\tlam, \omega)$ by \cite[Proposition 1]{Na06}.  So whether $K_{[\lambda(j)] \omega} > 0$ can be checked in polynomial time in terms of $\mathrm{size}([\lambda(j)], \omega)$.

To check whether $K_{[\lambda(j)] \omega} = 1$, we first replace $\omega$ by a partition $\mu \in \mcP_n$.  As in the proof of Corollary \ref{K1comp}, we may follow an algorithm to find a single set of indices which satisfies the conditions of Theorem \ref{MultiMultOne}, while performing the calculations required to check that $\widetilde{\lambda^k} \unrhd \mu^k$ for only this set of indices.  This is on the same order of computation time as checking whether $\tilde{\lambda} \unrhd \mu$.  The algorithm is essentially the same as in the proof of Corollary \ref{K1comp}, checking whether the subpartitions $[\lambda(j)^k]$ are of the desired shape when finding a potential index.  Following such an algorithm allows us to check whether $K_{[\lambda(j)] \mu} = 1$ in time $O(\mathrm{size}([\lambda(j)], \mu))$.  When we replace $\mu$ by a composition $\omega$, then considering the computation to find the permutation which maps $\omega$ to $\mu$, the total computation time is $O(\mathrm{size}([\lambda(j)], \mu) \ln(\mathrm{size}([\lambda(j)], \mu)))$.
\end{proof}

It follows from Theorem \ref{MultiMultOne} that for any $[\lambda(j)] \in \mcP_n[r]$, we have $K_{[\lambda(j)] \tlam} = 1$.  We now give a generalization of Corollary \ref{J1partitions}, by classifying which multipartitions $[\lambda(j)]$ satisfy $K_{[\lambda(j)] \mu} = 1$ for only $\mu = \tlam$.

\begin{corollary} \label{J1Multi}  Let $[\lambda(j)] \in \mcP_n[r]$.  The only $\mu \in \mcP_n$ satisfying $K_{[\lambda(j)] \mu} = 1$ is $\mu = \tlam$ if and only if for each $i$ either (1) $\tlam_i - \tlam_{i+1} \leq 1$, or (2) there exist at least two $j$ such that $\lambda(j)_i - \lambda(j)_{i+1} \geq 1$ (where we take $\tlam_i = 0$ if $i > \ell(\tlam)$ and $\lambda(j)_i = 0$ if $i > \ell(\lambda(j))$).
\end{corollary}
\begin{proof}  First suppose that there exists some $i =i'$, which we fix, such that $\tlam_{i'} - \tlam_{i'+1} > 1$ and $\lambda(j)_{i'} - \lambda(j)_{i'+1} \geq 1$ for at most one $j$.  This implies that $\lambda(j)_{i'} - \lambda(j)_{i'+1} \geq 1$ for a unique $j=j'$, and so $\tlam_{i'} - \tlam_{i'+1} = \lambda(j')_{i'} - \lambda(j')_{i'+1} > 1$.  We will construct a $\mu \in \mcP_n$ such that $\mu \neq \tlam$ and $K_{[\lambda(j)] \mu} = 1$.  Define $\mu$ to have parts $\mu_i = \tlam_i$ if $i \neq i', i'+1$, $\mu_{i'} = \tlam_{i'} -1$, and $\mu_{i'+1} = \tlam_{i'+1} + 1$.
Note that $\mu$ is a well-defined partition since $\tlam_{i'} - \tlam_{i'+1} \geq 2$.  If $\ell(\tlam) = l$, then $\ell(\mu) = l$ by definition.  Choose indices as in Theorem \ref{MultiMultOne} so that $t = l-1$,  with $\{i_k \, \mid \, k=1, \ldots, l-1\} = \{1, \ldots, i'-1, i'+1, i'+2, \ldots, l \}$.  Then for $k = 1, \ldots, l-1$, $[\lambda(j)^k]$ is given by
$$ \lambda(j)^k = \left\{ \begin{array}{lll} (\lambda(j)_k) & \text{ if } k < i', \\ (\lambda(j)_k, \lambda(j)_{k+1}) & \text{ if } k = i', \\ (\lambda(j)_{k+1}) & \text{ if } k > i', \end{array} \right.$$
and $\mu^k$ is given by
$$ \mu^k = \left\{ \begin{array}{ll} (\mu_k) & \text{ if } k < i', \\ (\mu_k, \mu_{k+1}) & \text{ if } k = i', \\ (\mu_{k+1}) & \text{ if } k > i'. \end{array}\right.$$
Now $\widetilde{\lambda^k} = \mu^k$ if $k \neq i'$, while $\mu^{i'} = (\tlam_{i'} - 1, \tlam_{i' +1} + 1) \neq \widetilde{\lambda^{i'}}$.
These indices satisfy the conditions of Theorem \ref{MultiMultOne} and so $K_{[\lambda(j)] \mu} = 1$, while $\mu \neq \tlam$.

We now suppose that for every $i$, either $\tlam_i - \tlam_{i+1} \leq 1$, or there are at least two $j$ such that $\lambda(j)_i - \lambda(j)_{i+1} \geq 1$, and we assume that $\mu \in \mcP_n$ satisfies $K_{[\lambda(j)] \mu} = 1$, and we must show $\mu = \tlam$.  Let $0 = i_0 < i_1 < \cdots < i_t = l$, where $\ell(\mu) = l$, be a choice of indices satisfying Theorem \ref{MultiMultOne}, with the accompanying $[\lambda(j)^k]$ and $\mu^k$ for $k= 1, \ldots, t$.    Let $[T(j)]$ be the unique semistandard Young $r$-multitableau of shape $[\lambda(j)]$ and weight $\mu$.  We know from the first paragraph of the proof of Theorem \ref{MultiMultOne} that $K_{[\lambda(j)^k] \mu^k} = 1$ for $k = 1, \ldots, t$, and we can consider each $[T(j)^k]$, the unique semistandard Young $r$-multitableau of shape $[\lambda(j)]^k$ and weight $\mu^k$.  It is enough to show that $\mu^k = \widetilde{\lambda^k}$ for each $k$.  Note that if $i$ is such that $\lambda(j)_i > \lambda(j)_{i+1}$ for at least two $j$, then we must have $i_k = i$ for some $k$, otherwise some $[\lambda(j)^k]$ will not satisfy the second condition of Theorem \ref{MultiMultOne}.  That is, for every $k$, we have that every $\lambda(j)^k$ has at most two distinct part sizes, and $\lambda(j)^k$ has exactly two distinct part sizes for at most one $j$.  By condition (1) above, these two part sizes can only differ by $1$.  For those $j$ where all part sizes of $\lambda(j)^k$ are equal, we know that in $T(j)$, each column of $\lambda(j)^k$ must have the entries $i_{k-1} + 1$ through $i_k$ in sequential order.  If every $j$ satisfies this, we have $\widetilde{\lambda^k} = \mu^k$ and we are done.  If $j'$ is such that $\lambda(j')^k$ has two distinct part sizes which differ by $1$, define $\mu^{k*}$ to have parts 
$$ \mu^{k*}_i = \mu^k_i - \sum_{j \neq j'} \lambda(j)^k_i.$$
Then we have $\lambda(j')^k \unrhd \mu^{k*}$, and $\lambda(j')^k$ has parts satisfying one of the conditions of Lemma \ref{ColLemma}, where distinct part sizes differ by $1$.  It follows from Lemma \ref{ColLemma} that $\lambda(j')^k = \mu^{k*}$, and so $\widetilde{\lambda^k} = \mu^k$. \end{proof}

\section{An Application to Finite General Linear Groups} \label{GLnq}

Let $\FF_q$ be a finite field with $q$ elements.  In this section we consider the group $G = \GL(n,\FF_q)$ of invertible $n$-by-$n$ matrices over $\FF_q$.  The complex irreducible characters of $G$ were first described by Green \cite{Gr55}.  We give a parameterization of the characters here which follows Macdonald \cite[Chapter IV]{Ma95}.

If $\mcO$ is a collection of finite sets, and $n$ is a non-negative integer, then an $\mcO$-multipartition of $n$ is a function $\blam: \mcO \rightarrow \mcP$ such that 
$$|\blam| = \sum_{\phi \in \mcO} |\phi| |\blam(\phi)| = n.$$
So, an $r$-multipartition can be viewed as an $\mcO$-multipartition if we take $\mcO = \{ \{1\}, \ldots, \{r\} \}$.  We let $\mcP_n^{\mcO}$ denote the set of all $\mcO$-multipartitions of $n$, and $\mcP^{\mcO} = \bigcup_{n \geq 0} \mcP_n^{\mcO}$.

Fix an algebraic closure $\bar{\FF}_q$ of $\FF_q$, and let $\bar{\FF}_q^{\times}$ be the multiplicative group of $\bar{\FF}_q$.  Then the Frobenius map $F$, defined by $F(a) = a^q$, acts on $\bar{\FF}_q^{\times}$, where the set of fixed points of $F^m$ is exactly $\FF_{q^m}^{\times}$, where $m>0$ is an integer.  Let $\hat{\FF}_{q^m}^{\times}$ denote the group of complex characters of $\FF_{q^m}^{\times}$.  When $k|m$, there is the standard norm map from $\FF_{q^m}^{\times}$ down to $\FF_{q^k}^{\times}$, which gives rise to the transposed norm map of character groups from $\hat{\FF}_{q^k}^{\times}$ to $\hat{\FF}_{q^m}^{\times}$.  We define $\mcX$ to be the direct limit of the character groups $\hat{\FF}_{q^m}^{\times}$ with respect to these norm maps, 
$$ \mcX = \lim_{\rightarrow} \hat{\FF}_{q^m}^{\times}.$$
The set $\mcX$ is in some since a dual of $\bar{\FF}_q^{\times}$, by taking consistent (non-canonical) bijections between each $\FF_{q^m}^{\times}$ and its character group.  The Frobenius map $F$ acts on $\mcX$ through the action on $\bar{\FF}_q^{\times}$, and we define $\Theta$ to be the collection of $F$-orbits of $\mcX$.  The irreducible complex characters of $G = \GL(n,\FF_q)$ are parameterized by the set $\mcP_n^{\Theta}$ of all $\Theta$-multipartitions of $n$.  Given $\blam \in \mcP^{\Theta}_n$, we let $\chi^{\blam} \in \Irr(G)$ denote the complex irreducible character of $\GL(n,\FF_q)$ to which it corresponds (not to be confused with characters of the symmetric group, which are not used in this section).
\\
\\
\noindent
{\bf Remark.  }  The set of orbits $\Theta$ is in cardinality-preserving bijection with the set of $F$-orbits of $\bar{\FF}_q^{\times}$, which are in turn in bijection with monic non-constant irreducible polynomials over $\FF_q$, where the degree of the polynomial is the same as the cardinality of the corresponding orbit.
\\

Let $U$ denote the group of unipotent upper triangular matrices in $G$, so 
$$ U = \{ (u_{ij}) \in G \, \mid \, u_{ii} = 1, u_{ij} = 0 \text{ if } i > j \}.$$
Fix a linear character $\theta: \FF_q^{+} \rightarrow \C^{\times}$ from the additive group of $\FF_q$ to the multiplicative group of complex numbers.  Given any $\mu \in \mcP_n$, define $I_{\mu}$ to be the complement in $\{1, 2, \ldots, n\}$ of the set of partial sums of $\mu$, so 
$$I_{\mu} = \{ 1, 2, \ldots, n \} \setminus \left\{ \sum_{i=1}^j \mu_i \, \mid \,  1 \leq j \leq \ell(\mu) \right\}  .$$
Define the linear character $\kappa_{\mu}$ of $U$ by
$$ \kappa_{\mu}((u_{ij})) = \theta \left( \sum_{ i \in I_{\mu} } u_{i, i+1} \right).$$
The \emph{degenerate Gel'fand-Graev character} $\Gamma_{\mu}$ of $G$ is now defined as
$$ \Gamma_{\mu} = \Ind_{U}^G(\kappa_{\mu}).$$

Gel'fand and Graev \cite{GeGr62} first considered $\Gamma_{\mu}$ for $\mu = (n)$, and proved its decomposition into irreducibles is multiplicity free.  Zelevinsky \cite[Theorem 12.1]{Zel81} decomposed $\Gamma_{\mu}$ for arbitrary $\mu$ as a linear combination of the irreducible characters of $G$, which in order to describe we must define another variant of the Kostka number.  Given $\blam \in \mcP_n^{\Theta}$, let $\Theta_{\blam} = \{ \varphi \in \Theta \, \mid \, \blam(\varphi) \neq \varnothing \}$, so $\Theta_{\blam}$ is the support of $\blam$.  Given $\mu \in \mcP_n$, a \emph{semistandard Young $\Theta$-multitableau of shape $\blam$ and weight $\mu$}, call it $\bT$, is a sequence of semistandard Young tableaux $[\bT(\varphi)]$ indexed by $\varphi \in \Theta_{\blam}$, such that each $\bT(\varphi)$ has shape $\blam(\varphi)$, and if $\omega(\varphi)$ is the weight of $\bT(\varphi)$, then the weight $\mu$ of $\bT$ has parts 
$$\mu_i = \sum_{\varphi \in \Theta_{\blam}} |\varphi| \omega(\phi)_i.$$
Then the Kostka number $K_{\blam \mu}$ is defined to be the total number of semistandard Young $\Theta$-multitableaux of shape $\blam$ and weight $\mu$.

\begin{theorem}[Zelevinsky] \label{ZelDecomp}
Given any $\mu \in \mcP_n$ and any $\chi^{\blam} \in \Irr(G)$, we have $\langle \Gamma_{\mu}, \chi^{\blam} \rangle = K_{\blam \mu}$.  That is, the degenerate Gel'fand-Graev character $\Gamma_{\mu}$ has the decomposition
$$ \Gamma_{\mu} = \sum_{\blam \in \mcP_n^{\Theta}} K_{\blam \mu} \chi^{\blam}.$$
\end{theorem}

Zelevinsky noted that if $\mu$ has parts $\mu_i = \sum_{\varphi \in \Theta} |\varphi| \blam(\varphi)_i$, then $K_{\blam \mu} = 1$, and he applied this to show that every irreducible character of $G$ has Schur index 1 \cite[Proposition 12.6]{Zel81}.

If $\blam \in \mcP_n^{\Theta}$, let $[\blam(\varphi)]$ denote the $|\Theta_{\blam}|$-multipartition indexed by $\varphi \in \Theta_{\blam}$ (where we fix some order of the elements of $\Theta_{\blam}$).  So if $|\Theta_{\blam}| = r$, then $[\blam(\varphi)] \in \mcP[r]$, and we may consider the partition $\tilde{\blam}$, where $|\tilde{\blam}| = |[\blam(\varphi)]|$.  If $w > 0$ is an integer, and $\mu \in \mcP_n$ is such that $w$ divides every part of $\mu$, then we let $\mu/w$ denote the partition with parts $\mu_i/w$.  Now if $\blam \in \mcP_n^{\Theta}$ has the property that $w$ divides $|\varphi|$ for every $\varphi \in \Theta_{\blam}$, then note that we have $|\tilde{\blam}| = |[\blam(\varphi)]| = |\mu/w|$.  By restricting our attention to those $\chi^{\blam}$ such that all $\varphi \in \Theta_{\blam}$ have the same size, we may apply our main results Theorem \ref{MultiMultOne} and Corollary \ref{J1Multi} to Theorem \ref{ZelDecomp} to obtain the following.

\begin{corollary} \label{ZelCor}
Let $\mu \in \mcP_n$, and suppose $\blam \in \mcP_n^{\Theta}$ is such that there exists an integer $w >0$ such that $|\varphi| = w$ for every $\varphi \in \Theta_{\blam}$, and say $|\Theta_{\blam}| = r$.  Then the following statements hold.
\begin{enumerate}
\item $\langle \Gamma_{\mu}, \chi^{\blam} \rangle = 1$ if and only if $[\blam(\varphi)] \in \mcP_{n/w}[r]$ and $\mu/w \in \mcP_{n/w}$ satisfy the conditions of Theorem \ref{MultiMultOne}.

\item Suppose $[\blam(\varphi)] \in \mcP_{n/w}[r]$ satisfies the conditions of Corollary \ref{J1Multi}.  Then $\langle \Gamma_{\mu}, \chi^{\blam} \rangle = 1$ if and only if $\mu/w = \tilde{\blam}$.
\end{enumerate}
\end{corollary}

The next natural problem is to understand the Kostka numbers $K_{\blam \mu}$ in the general case, when the orbits $\varphi \in \Theta_{\blam}$ have arbitrary sizes.  This appears to be a much more difficult problem, which is perhaps reflected when considering complexity.  As we see now, even in a simple case, the computational complexity of checking whether $K_{\blam \mu} >0$ is in stark contrast with Corollary \ref{Kmulticomp}.

\begin{proposition} \label{NP} Let $\blam \in \mcP_n^{\Theta}$ and $\mu \in \mcP_n$.  The problem of determining whether $K_{\blam \mu} > 0$, even if $\ell(\mu) = 2$ and $\blam(\varphi) = (1)$ for all $\varphi \in \Theta_{\blam}$, is $NP$-complete.
\end{proposition}
\begin{proof} To show that the general problem is $NP$, consider as a certificate a semistandard Young $\Theta$-multitableau $\bT$ of shape $\blam$ with entries $1$ through $\ell(\mu)$, which has storage size which is polynomial in $\mathrm{size}(\blam, \mu)$ (see \cite[Section 1]{PaVa10}).  If $\omega(\varphi)_i$ is the number of $i$'s in $\bT(\varphi)$, to check that $\bT$ has weight $\mu$ and thus verifying that $K_{\blam \mu} > 0$, we must check if $\sum_{\varphi \in \Theta_{\blam}} |\varphi| \omega(\varphi)_i = \mu_i$ for every $i$.  Since this can be checked in time which is polynomial in $\mathrm{size}(\blam, \mu)$, the problem is $NP$.

As remarked above, the orbits in $\Theta$ are in bijection with non-constant monic irreducible polynomials over $\FF_q$, where the cardinality of $\varphi \in \Theta$ is the degree of the corresponding polynomial.  It follows that as we increase $q$ and $n$, the sizes of the orbits in $\Theta$ can be any positive integer, and these sizes might occur with multiplicity as large as we like.  Now consider the case that $\blam(\varphi) = (1)$ for all $\varphi \in \Theta_{\blam}$ (in which case $\chi^{\blam} \in \Irr(G)$ is a \emph{regular semisimple} character), and $\ell(\mu) = 2, \mu = (\mu_1, \mu_2)$.  Then deciding whether $K_{\blam \mu} > 0$ is the same as determining whether the multiset of positive integers $\{ |\varphi| \, \mid \, \varphi \in \Theta_{\blam} \}$ has some sub-multiset which has sum $\mu_1$.  This is exactly the \emph{subset sum} problem, which is known to be $NP$-complete (see \cite[Exercise 2.17]{ArBa09}, for example).  Thus to determine whether $K_{\blam \mu} > 0$ in general is $NP$-complete.
\end{proof}

While Proposition \ref{NP} does not prevent there from being nice theoretical answers, it does indicate that the problem quickly gets more complicated than the questions we have addressed so far.  We conclude by leaving the problem of further understanding the numbers $K_{\blam \mu}$ as an open question.

\bigskip

\noindent
\begin{tabular}{ll}
\textsc{Department of Mathematics}\\ 
\textsc{College of William and Mary}\\
\textsc{P. O. Box 8795} \\
\textsc{Williamsburg, VA  23187}\\
{\em e-mail}:  {\tt jrjanopaulnayl@email.wm.edu}, {\tt vinroot@math.wm.edu}\\
\end{tabular}


\begin{thebibliography}{10}

\bibitem{ArBa09}
S. Arora and B. Barak, Computational complexity, a modern approach, Cambridge University Press, Cambridge, 2009.

\bibitem{BeZe90}
A. D. Berenshte\u{\i}n and A. V. Zelevinski\u{\i}, When is the multiplicity of a weight equal to 1?, \emph{Funct. Anal. Appl.} \textbf{24} (1990), no. 4, 259--269.

\bibitem{Fu97}
W. Fulton, Young tableaux, with applications to representation theory and geometry, London Mathematical Society Student Texts, 35, Cambridge University Press, Cambridge, 1997.

\bibitem{FuHa91}
W. Fulton and J. Harris, Representation theory, a first course, Graduate Texts in Mathematics, 129, Readings in Mathematics, Springer-Verlag, New York, 1991.

\bibitem{GaGoVi12}
Z. Gates, B. Goldman, and C. R. Vinroot, On the number of partition weights with Kostka multiplicity one, \emph{Electron. J. Combin.} \textbf{19} (2012), no. 4, Paper 52, 22 pp.

\bibitem{GeGr62}
I. M. Gel'fand and M. I. Graev, Construction of the irreducible representations of simple algebraic groups over a finite field, \emph{Dokl. Akad. Nauk SSSR} \textbf{147} (1962), 529--532.

\bibitem{Gr55}
J. A. Green, The characters of the finite general linear groups, \emph{Trans. Amer. Math. Soc.} \textbf{80} (1955), no. 2, 402--447.


\bibitem{InJiSt09}
F. Ingram, N. Jing, and E. Stitzinger, Wreath product symmetric functions, \emph{Int. J. Algebra} \textbf{3} (2009), no. 1-4, 1--19.

\bibitem{JaKe81}
G. D. James and A. Kerber, The representation theory of the symmetric group, Addison-Wesley, Reading, Mass., 1981.

\bibitem{Ma95}
I. G. Macdonald, Symmetric functions and Hall polynomials, second edition, with contributions by A. Zelevinsky, Oxford Mathematics Monographs, Oxford Science Publications, The Clarendon Press, Oxford University Press, New York, 1995.


\bibitem{Mar12}
E. Marberg, Generalized involution models for wreath products, \emph{Israel J. Math.} \textbf{192} (2012), no. 1, 157--195.

\bibitem{Na06}
H. Narayanan, On the complexity of computing Kostka numbers and Littlewood-Richardson coefficients, \emph{J. Algebraic Combin.} \textbf{24} (2006), no. 3, 347--354.

\bibitem{PaVa10}
I. Pak and E. Vallejo, Reductions of Young tableau bijections, \emph{SIAM J. Discrete Math.} \textbf{24} (2010), no. 1, 113--145.


\bibitem{Sa90}
B. Sagan, The ubiquitous Young tableau, In: Invariant theory and tableaux (Minneapolis, MN, 1988), 262--298, IMA Vol. Math. Appl., 19, Springer, New York, 1990.

\bibitem{Sp32}
W. Specht, Eine Verallgemeinerung der symmetrischen Gruppe, \emph{Schriften Math. Seminar (Berlin)} \textbf{1} (1932), 1--32.

\bibitem{Zel81}
A. V. Zelevinsky, Representations of finite classical groups, a Hopf algebra approach, Lecture Notes in Mathematics, 869, Springer-Verlag, Berlin-New York, 1981.

\end{thebibliography}
\end{document}